\documentclass[11pt]{article}
\usepackage{graphicx, amsmath, amsfonts, amssymb, amsthm, fullpage, authblk, bbm, setspace}
\usepackage[dvipsnames]{xcolor}
\usepackage{algorithm}
\usepackage{algpseudocode}
\algnewcommand\algorithmicinput{\textbf{INPUT:}}
\algnewcommand\INPUT{\item[\algorithmicinput]}
\algnewcommand\algorithmicoutput{\textbf{OUTPUT:}}
\algnewcommand\OUTPUT{\item[\algorithmicoutput]}

\usepackage{hyperref}[]
\hypersetup{
    colorlinks=true,
    linkcolor=blue,
    filecolor=magenta,      
    urlcolor=cyan,
      citecolor=blue,
    }
\usepackage{cleveref}

\newtheorem{theorem}{Theorem}
\newtheorem{lemma}[theorem]{Lemma}
\newtheorem{proposition}[theorem]{Proposition}
\newtheorem{corollary}[theorem]{Corollary}

\newtheorem{definition}{Definition}

\newtheorem{assumption}{Assumption}

\allowdisplaybreaks
\newcommand{\dint}{\,\mathrm{d}}

\usepackage[round]{natbib}

\title{Differentially private hypothesis testing in survival analysis}
\author{Elly K.~H.~Hung}
\author{Yi Yu}
\affil{Department of Statistics, University of Warwick}
\date{}
\begin{document}

\maketitle

\begin{abstract}
Survival analysis is widely used in applications involving sensitive individual-level data, yet differentially private hypothesis testing for right-censored data remains largely undeveloped. We initiate a finite-sample theory of private hypothesis testing in survival analysis applications. For Cox regression coefficients, we develop private partial-likelihood-ratio and score-type tests, including a private calibration procedure for the rejection threshold. For cumulative hazard functions, we propose a private distributed two-sample test. Across these problems, we prove differential privacy and finite-sample testing guarantees, as well as minimax lower bounds. Our results identify when privacy is statistically negligible, when it dominates the testing rate, and where optimal private rates for testing in semiparametric survival models remain open. This theoretical analysis is accompanied by numerical experiments on simulated data. 

\end{abstract}

\section{Introduction}
Survival analysis underpins decision-making in many high-stakes applications, including medicine \citep[e.g.][]{cox1972regression}, engineering reliability \citep[e.g.][]{meeker2021reliability}, and customer retention \citep[e.g.][]{fader2007project}, where the data often contain sensitive individual-level information. However, it has been shown that individuals may be re-identified even from seemingly anonymized datasets \citep[e.g.][]{narayanan2008deanonymization, rocher2019reidentification}. Differential privacy \citep{dwork2006} provides a rigorous mathematical framework for protecting against such disclosure.

\begin{definition}[$(\epsilon, \delta)$-differential privacy] \label{def-cdp} 
For $\epsilon > 0$ and $\delta \geq 0$, a privacy mechanism $M$ satisfies $(\epsilon, \delta)$-differentially private if it is a conditional distribution on a space $\mathcal{R}$ such that
\begin{equation*}
    M(R \in A \mid D) \leq e^\epsilon M(R \in A \mid D') + \delta,
\end{equation*}
for any measurable set $A$ in the sigma-algebra of $\mathcal{R}$ and any pair of datasets $(D, D')$ differing in at most one data point.
\end{definition}

While there has been increasing attention towards differentially private estimation procedures for survival analysis 
\citep[e.g.][]{Nguyen2017DPsurvival, pmlr-v126-gondara20a, EgeaEscobar2025survival, FDPCox}, private hypothesis testing for right-censored data remains much less developed. The only prior work we are aware of is \cite{pmlr-v126-gondara20a}, which considers the empirical performance of a test statistic derived from a privatized discrete-time Kaplan--Meier estimator. Most of the private testing literature assumes independent likelihood decompositions, which does not hold in many survival analysis applications due to the observations being linked through at-risk sets, so the standard approaches cannot be directly applied.

We initiate a finite-sample study of differentially private hypothesis testing for several canonical problems in survival analysis. First, for binary testing of Cox regression coefficients, we construct a private partial likelihood ratio test and prove a detection guarantee, together with a minimax lower bound that identifies the privacy-dependent cost and shows a remaining dimension-dependent gap. Second, for Cox coefficient identity testing with a simple null and composite alternative, we develop a private score-type test based on the Euclidean norm of the score, avoiding private estimation of the inverse information matrix. The rejection threshold is estimated by a private calibration procedure. Third, for two-sample testing of cumulative hazard functions under right censoring, we propose a distributed private test and prove that its separation rate is minimax optimal up to poly-logarithmic factors. Finally, we show the empirical performance of our proposed tests in a simulation study. 

\subsection{Related literature}
Although hypothesis testing under differential privacy (\Cref{def-cdp}) has received less attention than estimation, interest in the area has grown in recent years. Parametric procedures that have been developed include tests for linear regression coefficients \citep[e.g.][]{sheffet2017differentially, barrientos2019regression, alabi2023regression}, for analysis of variance \citep[e.g.][]{campbell2018dpanova}, and for estimating $p$-values based on bounded-influence M-estimators \citep[e.g.][]{avella2020privacy}. Differentially private analogues of classical nonparametric tests have also been proposed, including $\chi^2$ tests \citep[e.g.][]{gaboardi2016dpchisq, rogers2017privatechisq}, rank-based tests \citep[e.g.][]{couch2019dpnonparametric}, and Kolmogorov--Smirnov tests \citep[e.g.][]{awan2025KS}. There has also been a closely related line of work on differentially private confidence intervals \citep[e.g.][]{ferrando2022DPCI, chadha2024resampling, wang2025dpbootstrap}. 

Most of the aforementioned literature adopts an asymptotic perspective. There have also been works on minimax separation rates, including for binary likelihood ratio tests \citep{canonne2019structure}, tests for binomial data \citep{awan2018binomial}, tests for discrete distributions \citep[e.g.][]{cai2017privit, acharya2018discretedists}, Gaussian mean testing \citep[e.g.][]{pmlr-v178-narayanan22a}, testing in the white-noise-with-drift model \citep{cai2024federatednonparametrichypothesistesting}, and permutation tests \citep{kim2026DPpermutation}. Building on the subsample-and-aggregate idea of \cite{cai2017privit}, \cite{pena2022subsample} and \cite{kazan2023test} study black-box approaches for converting non-private tests into differentially private procedures, although these methods often have sub-optimal sample complexity.

The literature discussed here, as well as this paper, is set in the central model of differential privacy. There has also been a body of research on hypothesis testing under local or other distributed privacy constraints, see, for example, Section 1 of \cite{JMLR:v26:24-2016} for a review.

\section{Tests for Cox regression coefficients}
\subsection{Model and assumptions}
\label{sec:assumptions}
In this section, we consider binary (\Cref{sec-2}) and identity (\Cref{sec-3}) hypothesis testing for the regression coefficients in a Cox proportional hazards model \citep{cox1972regression}. The Cox model is one of the most widely used models in survival analysis, and assumes that conditional on a covariate process $\{Z(t)\}_{t\in[0,1]}$, the hazard rate of the event time $\widetilde T$ is
\begin{equation}\label{eq:hazard}
\lambda(t) = \lambda_0(t)\exp\{\beta^{*\top}Z(t)\},
\end{equation}
where $\beta^*\in\mathbb{R}^d$ is a vector of unknown regression coefficients and $\lambda_0(\cdot)$ is an unknown baseline hazard function. We observe right-censored data $(T, \Delta, Z)$, where $T = \min\{\widetilde{T}, C\}$ and $\Delta = \mathbbm{1}\{\widetilde{T} \leq C\}$ for a censoring time $C$. 

We adopt the counting process representation of \cite{andersen1982}.  For $i \in \{1, 
\ldots, n\}$, define the processes $N_i(t) = \mathbbm{1}\{T_i < t,  \Delta_i = 1\}$ and $Y_i(t) = \mathbbm{1}\{N_i(t) \geq t\}$ for $t \in [0, 1]$. For any $\beta \in \mathbb{R}^d$, the log partial likelihood for the Cox model is therefore
\begin{equation}
\label{eq:Cox_partial_likelihood}
    \ell_n(\beta) = \sum_{i=1}^n \int_0^1 \beta^\top Z_i(t) \dint N_i(t) - \sum_{i=1}^n \int_0^1 \log \left( \sum_{j=1}^n Y_j(t) \exp\{\beta^\top Z_j(t)\}\right) \dint N_i(t),
\end{equation}
and its gradient is
\begin{equation} \label{eq:Cox_grad}
    \dot{\ell}_n(\beta) = \sum_{i=1}^n \int_0^1 \{Z_i(t) - \bar{Z}(t, \beta) \} \dint N_i(t), \mbox{ with } \bar{Z}(t, \beta) = \frac{\sum_{i=1}^n Z_i(t)Y_i(t) \exp(\beta^\top Z_i(t))}{\sum_{i=1}^n Y_i(t) \exp(\beta^\top Z_i(t))}.
\end{equation}

We assume that $\{(\widetilde{T}_i, C_i, \{Z_i(t), \,t \in [0, 1]\})\}_{i=1}^n$ are independent and identically distributed copies of $(\widetilde{T}, C, \{Z(t), \,t \in [0, 1]\})$, where the conditional hazard of $\tilde{T_i}$ given $Z_i$ is as specified in \eqref{eq:hazard}. We collect further assumptions used in our results below. 

\begin{assumption} \label{assp:covariates} Assume that $\{Z(t), \,t \in [0, 1]\}$ is predictable, and there exist absolute constants $C_Z, L_Z >0$ such that 
\begin{equation*}
     \mathbb{P}\left(\sup_{t \in [0, 1]} \|Z(t)\|_2 \leq C_Z \right) = 1 \quad \mathrm{and} \quad \mathbb{P} \left( \sup_{0 \leq s \leq t \leq 1} \|Z(s) - Z(t)\|_2 \leq L_Z|s- t| \right) = 1.
\end{equation*}
\end{assumption}

\begin{assumption}\label{assp:indep_atrisk} Assume that the event time $\widetilde{T}$ and censoring time $C$ are conditionally independent given the covariate process $Z$. There exist absolute constants $p_0, C_\lambda > 0$ such that $\mathbb{P}(Y(1)=1) \geq p_0$ and $\lambda_0(t) \leq C_\lambda$ for all times $t \in [0, 1]$. 
\end{assumption}

\begin{assumption}\label{assp:eigenvalue} Define a population version of $-\ddot{\ell}_n(\beta^*)/n$ as 
\begin{equation} \label{eq-G-def}
G(\beta^*) = \mathbb{E} \left[ \int_0^1 Y(t) \exp\{\beta^{*\top} Z(t)\}\{Z(t) - \mu(t, \beta^*)\}^{\otimes 2} \dint \Lambda_0(t) \right], 
\end{equation}
where $\mu(t, \beta) = \mathbb{E}[Z(t) Y(t) \exp\{\beta^\top Z(t)\}]/\mathbb{E}[Y(t) \exp\{\beta^\top Z(t)\}]$ and we denote $v^{\otimes 2} = vv^\top$.  For some absolute constants $\rho_+ \geq \rho_- > 0$, assume that the eigenvalues of $G(\beta^*)$ satisfy $\rho_-/d \leq \lambda_{\min}(G(\beta^*)) \leq \lambda_{\max}(G(\beta^*)) \leq \rho_+/d$. 
\end{assumption}
For any $\beta \in \mathbb{R}^d$, we will use $\mathcal{P}(\beta)$ to denote the set of distributions on $[0, 1] \times \{0, 1\} \times \mathbb{R}^d$ which satisfy Assumptions~\ref{assp:covariates}, \ref{assp:indep_atrisk}, and~\ref{assp:eigenvalue}, and where the failure times are generated according to the model specified by \eqref{eq:hazard}.  
\Cref{assp:covariates} imposes boundedness and smoothness on the covariate process. Similar boundedness conditions are frequently imposed in the differential privacy literature to calibrate the level of noise needed to preserve privacy.  \Cref{assp:indep_atrisk} is the usual non-informative censoring condition in survival analysis, together with mild positivity and boundedness conditions on the at-risk probability at the end of the study and the baseline hazard function. \Cref{assp:eigenvalue} can be seen as requiring the expected information matrix to be uniformly well conditioned, to ensure sufficient curvature for identification.

\subsection{Binary hypothesis testing}
\label{sec-2}
We begin with binary hypothesis testing for the regression coefficient in the Cox model \eqref{eq:hazard}, i.e.
\begin{equation}
        \mathrm{H}_0: \beta^* = \beta_0 \quad \mbox{vs.} \quad \mathrm{H}_1: \beta^* = \beta_1.
        \label{eq:binary_test_def}
\end{equation}

For binary hypothesis testing without privacy constraints, the likelihood ratio test is optimal by the Neyman--Pearson lemma \citep[e.g.][Theorem 8.3.12]{casella2024statistical}. The standard asymptotic results can also be applied to the Cox partial likelihood for inference on $\beta^*$ \citep{therneau2000cox}. As a natural starting point, we propose a differentially private version of the partial likelihood ratio test, and state its detection guarantee in \Cref{prop:beta_binary_upper}.

\begin{proposition} \label[proposition]{prop:beta_binary_upper}
    Assume that Assumptions~\ref{assp:covariates}, \ref{assp:indep_atrisk} and \ref{assp:eigenvalue} hold. For any $\beta_0, \beta_1 \in \mathbb{B}_{C_\beta}(0)$, with $\beta_0 \neq \beta_1$ and $C_{\beta} > 0$ being an absolute constant, consider the hypothesis test in \eqref{eq:binary_test_def}.
    Define the log likelihood ratio test as $\phi = \mathbbm{1}\{ \gamma(\beta_0, \beta_1) <0\}$, where 
    \begin{equation*}
        \gamma(\beta_0, \beta_1) = \ell_n(\beta_0) - \ell_n(\beta_1) + \frac{c_{\beta_0, \beta_1} (1+\log(n))\|\beta_0 - \beta_1\|_2}{\epsilon} W, \quad W \sim \mathrm{Laplace}(\mu=0, b=1),
    \end{equation*}
    with $c_{\beta_0, \beta_1}= 4C_Z + \exp(2\max\{\|\beta_0\|_2, \|\beta_1\|_2\} C_Z)(2C_Z + C_Z^2)$. We have that (i) the test $\phi$ is $(\epsilon, 0)$-differentially private; and (ii) there exists absolute constants $n_0, C_1, C_2 > 0$ are absolute constants such that for any $n > n_0$, provided that $d \log(nd)^2 \leq C_1 \sqrt{n}$ and 
    \begin{equation}
        \|\beta_1 - \beta_0\|^2_2 \geq C_2 \left(\frac{d}{n} + \frac{d^2(1+\log(n))^2}{n^2\epsilon^2}  \right),
        \label{eq:beta_binary_rate}
    \end{equation}
    it holds that $\mathbb{P}_{\mathrm{H}_0}(\phi=1) + \mathbb{P}_{\mathrm{H}_1}(\phi=0) < 1/3$.
\end{proposition}
The test in \Cref{prop:beta_binary_upper} uses the Laplace mechanism \citep[e.g.][Definition 3.3]{dwork2014algorithmic}, which adds Laplace noise calibrated to the \emph{sensitivity} \citep[e.g.][Definition 3.1]{dwork2014algorithmic} of the log partial likelihood ratio. 
\Cref{prop:beta_binary_upper} is stated for pure differential privacy, and therefore satisfies $(\epsilon, \delta)$-differential privacy for all $\delta \geq 0$. There are regimes of $\delta>0$ where replacing the Laplace perturbation by the Gaussian mechanism \citep[e.g.][Theorem 3.22]{dwork2014algorithmic} calibrated with the same sensitivity may lead to smaller constant factors in the rate. 

Furthermore, \Cref{prop:beta_binary_upper} provides a sufficient condition on the separation between the two hypotheses under which reliable testing is possible.  The upper bound 1/3 is arbitrary and can be replaced by any absolute constant in $(0, 1)$. The first term of \eqref{eq:beta_binary_rate}, $d/n$, is the usual non-private rate for simple hypothesis testing. The second term quantifies the additional difficulty induced by the privacy-presrving noise. In particular, the privacy constraint has no effect on the testing rate when $\epsilon \gtrsim \sqrt{d/n}$, whereas for smaller $\epsilon$ the privacy cost dominates.

We now consider a lower bound on the minimax separation rate. Define the minimax separation rate in the binary hypothesis testing problem as 
\begin{align} \label{eq-separation-def-bi}
& r^*(d, n, \epsilon, \delta) = \inf \Big\{ r > 0 \ : \ \forall \beta_0, \beta_1 \in \mathbb{R}^d: \|\beta_0 - \beta_1\|_2 > r,\, \forall \mathbb{P}_{\beta_0} \in \mathcal{P}(\beta_0) \text{ and } \mathbb{P}_{\beta_1} \in \mathcal{P}(\beta_1) \nonumber \\
& \hspace{3.5cm} \exists \ (\epsilon, \delta)\text{-differentially private test } \phi \text{ s.t.~} \mathbb{P}_{\beta_0}^{\otimes n}(\phi=1) + \mathbb{P}_{\beta_1}^{\otimes n}(\phi=0) < 1/3 \Big\}.
\end{align}

\begin{proposition}
\label[proposition]{prop:beta_binary_lower} There exists an absolute constant $c>0$ such that for all $\epsilon>0, \delta \geq 0, n, d \in \mathbb{N}$, for the separation rate defined in \eqref{eq-separation-def-bi}, it holds that 
\begin{equation*}
    r^*(d, n, \epsilon, \delta) \geq c \left\{\frac{d}{n} + \frac{d}{n^2 (\epsilon + \delta)^2} \right\}.
\end{equation*}
\end{proposition}

Propositions~\ref{prop:beta_binary_upper} and \ref{prop:beta_binary_lower} imply that, in the usual regime of $\delta \lesssim \epsilon$, we have that
\[
    \frac{d}{n}+ \frac{d}{n^2 \epsilon^2} \lesssim r^*(d, n, \epsilon, \delta) \lesssim \frac{d}{n} + \frac{d^2}{n^2\epsilon^2},
\]
which shows an unresolved gap where the optimality of private testing based on the Cox partial likelihood remains open. The proofs of these propositions are deferred to \Cref{app:sec-2}. 

For the general binary hypothesis testing problem, \cite{canonne2019structure} derive an optimal test based on a randomized truncated log-likelihood ratio test. This relies on the log-likelihood decomposing as a sum of independent terms, which is absent in the Cox partial log-likelihood, due to the dependence induced by the at-risk sets. This suggest that the optimal rate for private binary testing for coefficients in the Cox model, or even more generally with unknown nuisance parameters, remains unknown.

Nevertheless, our approach improves upon black-box subsample-and-aggregate methods, as in e.g.~\cite{pena2022subsample} and \cite{kazan2023test}, where the sample complexity is inflated by a factor of $O(1/\epsilon)$ compared to its non-private counterpart. Since the optimal non-private separation rate is $\|\beta_1-\beta_0\|^2_2 \asymp d/n$, such methods would require $\|\beta_1-\beta_0\|^2_2\gtrsim d/(n \epsilon)$, which is slower than the rate achieved by \Cref{prop:beta_binary_upper} for $\epsilon \lesssim 1$.

\subsection{Private identity testing for regression coefficients}\label{sec-3}
In this subsection, we consider the simple versus composite hypothesis testing problem of
\begin{equation} \label{eq-test-good-reg}
    \mathrm{H}_0: \beta^* = \beta_0 \quad \mbox{vs.} \quad \mathrm{H}_1: \|\beta^* - \beta_0\|_2^2 \geq r^2.
\end{equation}
 Our analysis proceeds in two steps. We study a differentially private score test statistic, first by assuming that the trace term $\mathrm{tr}\{G(\beta^*)\}$ used for calibrating the rejection threshold is available. Here, $G(\cdot)$ is a population version of $-\ddot{\ell}_n(\beta^*)/n$, which is defined in \eqref{eq-G-def}, and $\beta^*$ denotes the underlying true parameter of the Cox model. We then construct a private estimator of $\mathrm{tr}\{G(\beta^*)\}$ and show that using it in place of the oracle quantity in the threshold of the test retains the same error guarantee up to constant and logarithmic factors. The proofs of the results in this subsection can be found in \Cref{app:sec-3}. 

\subsubsection{Identity testing via a private score statistic}
In the non-private setting, score tests are a natural choice for testing a simple null on the regression coefficient, since they only require evaluating the gradient of the log partial likelihood at the null value~$\beta_0$. Under $\mathrm{H}_0$, the score is centred, whereas under $\mathrm{H}_1$ its magnitude becomes systematically large. Motivated by this, we first study an idealized private score test in which the true trace term $\mathrm{tr}\{G(\beta^*)\}$ is available to calibrate the rejection threshold, before providing a private estimator of $\mathrm{tr}\{G(\beta^*)\}$ in \Cref{sec-private-trace}. 

\begin{proposition} \label[proposition]{lemma:scoretest_theoretical}
Assume that Assumptions~\ref{assp:covariates}, \ref{assp:indep_atrisk} and \ref{assp:eigenvalue} hold.  For $\beta_0, \beta^* \in \mathbb{B}_{C_\beta}(0) \subset \mathbb{R}^d$, consider the hypothesis test \eqref{eq-test-good-reg}.  Suppose that $\mathrm{tr}\{G(\beta^*)\}$ is available. Define the test
\begin{equation} \label{eq-comp-test-cox}
    \phi(\beta_0) = \mathbbm{1}\bigg\{ \bigg\|\frac{\dot{\ell}_n(\beta_0)}{\sqrt{n}}\bigg\|_2 + \frac{C_{\beta_0}(1+\log(n))}{\sqrt{n}\epsilon}W > \tau \bigg\}, \quad W\sim \mathrm{Laplace}(\mu=0,b=1),
\end{equation}
where $C_{\beta_0} = 4C_Z + \exp(2C_Z\|\beta_0\|_2)(2C_Z+C_Z^2)$. Then for any choice of $\tau \in \mathbb{R}$, $\phi(\beta_0)$ is $(\epsilon, 0)$-differentially private. Further, there exists absolute constants $C_1, C_2, c_1, c_2$ that depend on the absolute constants in \Cref{assp:eigenvalue} such that if $\max\{ \log(nd)^4d^2/n, d^{-1/2}\} \leq C_1$ and the separation in \eqref{eq-test-good-reg} satisfies
\begin{equation}
\label{eq:beta_composite_assumption}
    r^2 \geq C_2 \log(nd)^4 \max \left\{ \frac{d^{3/2}}{n}, \frac{d^4}{n^2}, \frac{d^2}{\min (n^2 \epsilon^2, n^{3/2} \epsilon)}\right\},    
\end{equation} 
choosing the threshold in the test $\phi(\beta_0)$ to be 
\begin{equation} \label{eq:beta_composite_threshold}
    \tau = \mathrm{tr}\{G(\beta^*)\}^{1/2} + c_1d^{-1/2} + c_2 C_{\beta_0}(1+\log(n))n^{-1/2}\epsilon^{-1},
\end{equation} 
we have that $\mathbb{P}_{\mathrm{H}_0}(\phi=1) + \mathbb{P}_{\mathrm{H}_1}(\phi=0) \leq 1/3$.
\end{proposition}
Rather than privatising the classical score statistic $\dot{\ell}_n(\beta_0)^\top \{V(\beta_0)\}^{-1}\dot{\ell}_n(\beta_0)$, we work directly with the Euclidean norm of the score and calibrate the rejection threshold by $\mathrm{tr}\{G(\beta^*)\}$. Our approach avoids the need to privately estimate the inverse of a $d$-dimensional matrix. Differentially private inference for the Cox model is particularly challenging because the log partial likelihood is formed from weighted sums over the at-risk sets, and hence its gradient and Hessian lack the simpler structure of independent and identically distributed terms. 

Existing work such as \cite{avella2020privacy} and \cite{avella2023differentially} studies the asymptotic validity of classical procedures in more standard M-estimation settings. In contrast, \Cref{lemma:scoretest_theoretical} gives an explicit finite-sample detection guarantee of a differentially private score-type test for \eqref{eq-test-good-reg}, by showing how the separation required for our procedure depends on the dimension, sample size and privacy level. There are three terms in the separation condition \eqref{eq:beta_composite_assumption}: the first is the non-private rate; the second term is from only imposing the eigenvalue conditions of \Cref{assp:eigenvalue} at the true $\beta^*$, and is dominated by the first term when $d^{5/2} \lesssim n$; and the third term is for having sufficient signal relative to the privacy-preserving noise. 

The question of whether the separation rate in \Cref{lemma:scoretest_theoretical} is optimal over all differentially private tests remains open.  We have the following non-private lower bound.

\begin{lemma} \label[lemma]{lemma:beta_composite_lower}
Fix any $\beta_0 \in \mathbb{R}^d$. Suppose we have $\{(T_i, \Delta_i, Z_i)\}_{i=1}^n$ drawn from some distribution in $\mathcal{P}(\beta^*)$ and consider the testing problem in \eqref{eq-test-good-reg}. There exists some absolute constant $c > 0$ such that if $\|\beta_0 - \beta^*\|_2^2 \leq c d^{3/2}/n$, then there is no test $\phi$ that satisfies $\mathbb{P}_{\mathrm{H}_0}(\phi=1) + \mathbb{P}_{\mathrm{H}_1}(\phi=0) < 1/3$. 
\end{lemma}

\Cref{lemma:beta_composite_lower} shows that the non-private component of the separation, namely the term of order~$d^{3/2}/n$, is unavoidable, so in this sense the non-private part of \Cref{lemma:scoretest_theoretical} is nearly minimax optimal. The optimal dependence of the privacy term is currently unknown. To the best of our knowledge, matching upper and lower bounds for simple-null composite-alternative testing under differential privacy constraints have only been established for a small number of settings, notably for discrete distributions \citep[e.g.][]{acharya2018discretedists}, for Gaussian mean testing \citep{pmlr-v178-narayanan22a}, and, more recently, for the white-noise-with-drift model \citep{cai2024federatednonparametrichypothesistesting}. In these works, the lower bounds rely on delicate coupling arguments, and it is not clear how to extend such constructions to the present survival setting, or even to simpler settings with a response variable. Establishing the optimal private separation rate for identity testing in the Cox model therefore remains an interesting open problem.

\subsubsection{Implementation via private trace estimation} \label{sec-private-trace}

The test in \Cref{lemma:scoretest_theoretical} is stated in terms of the population quantity $\mathrm{tr}\{G(\beta^*)\}$ and is therefore not directly implementable. We now show that one can replace this oracle quantity by a differentially private estimator without changing the separation needed for the testing guarantee beyond constant and logarithmic factors. To do so, we consider the negative normalized Hessian of the partial log-likelihood as an empirical version of $G(\beta_0)$: 
\begin{equation} \label{eq:truncated_Hessian}
    H(D; \beta_0) = \frac{-\ddot{\ell}_n(\beta_0)}{n} =\frac{1}{n} \sum_{i=1}^n \int_0^1 \frac{\sum_{j=1}^n Y_j(t) \exp(\beta_0^\top Z_j(t))\{Z_j(t) - \bar{Z}(t, \beta_0)\}^{\otimes 2}}{ \sum_{j=1}^n Y_j(t) \exp(\beta_0^\top Z_j(t))} \dint N_i(t). 
\end{equation}
We show that the sensitivity of $\mathrm{tr}\{H(D; \beta_0)\}$ is of order $\log(n)/n$ in \Cref{lemma:trace_sensitivity} of \Cref{sec-trace-sensitivity}. This allows us to construct a private estimator of $\mathrm{tr}\{G(\beta_0)\}$ and we show the theoretical guarantees of using this in place of the true $\mathrm{tr}\{G(\beta^*)\}$ in the rejection threshold \eqref{eq:beta_composite_threshold}.  
\begin{corollary} \label[corollary]{lemma:trace_plugin}
Let $D_1=\{(T_{1,i}, \Delta_{1,i}, Z_{1,i})\}_{i=1}^{n/2}$ and $D_2=\{(T_{2,i}, \Delta_{2,i}, Z_{2,i})\}_{i=1}^{n/2}$ be i.i.d.~and $D=D_1 \cup D_2$. Let $H(D_1; \beta_0)$ be as defined in \eqref{eq:truncated_Hessian} and define an estimator for its trace as  \begin{equation}\label{eq:trace_estimator_def}
        T(D_1; \beta_0)= \max\left\{0, \mathrm{tr}\{H(D_1; \beta_0)\} + \frac{K(n/2, \beta_0)}{\epsilon} W'\right\},   \quad W' \sim \mathrm{Laplace}(\mu=0, b=1)
\end{equation}
where $K(n, \beta_0) \asymp \log(n)/n$ as defined in \eqref{eq:trace_sensitivity} in \Cref{sec-trace-sensitivity}. Under the same assumptions and settings as in \Cref{lemma:scoretest_theoretical} (up to constants and logarithmic factors), replacing $\mathrm{tr}\{G(\beta^*)\}$ in \eqref{eq:beta_composite_threshold} with $T(D_1; \beta_0)$, the privacy guarantees with respect to $D$ and the test error controls in \Cref{lemma:scoretest_theoretical} still hold for the test $\phi$ constructed using $D_2$ and noise $W \perp W'$.
\end{corollary}

\Cref{lemma:trace_plugin} shows that the oracle test of \Cref{lemma:scoretest_theoretical} is not merely a conceptual benchmark and provides a data-driven procedure that retains the same guarantee up to constants and logarithmic factors. Our construction can be generalized beyond the Cox model; its main ingredients are a score-type statistic whose non-private fluctuation can be controlled by a scalar calibration term, and an empirical analogue of that calibration term with global sensitivity of order $n^{-1}$ (up to log factors). Whenever this structure is available, the same strategy can be used to turn an oracle private test into a fully implementable differentially private procedure. In particular, the argument should extend to other semi-parametric or generalized linear models with bounded covariates and suitably controlled variance weights.

\section{Private two-sample testing for cumulative hazards}\label{sec-4}
In this section, we consider a distributed two-sample testing problem for right-censored data, where the samples are on two servers. For $k \in \{1, 2\}$,  Server $k$ holds the dataset $D_k = \{(T_{k,i}, \Delta_{k,i})\}_{i=1}^{n_k}$, where the event times $\tilde{T}_{k, i}$'s are independent and identically distributed from a distribution with cumulative hazard $\Lambda_k(t) = \int_0^t \lambda_k(s) \dint s$. We continue to assume that the distributions for the event and censoring times satisfy \Cref{assp:indep_atrisk} from \Cref{sec:assumptions} in a covariate-less setting, and denote this set of distributions by $\mathcal{C}$.  Our goal is to test 
\begin{equation} \label{eq-two-sample-test}
    \mathrm{H}_0: \Lambda_1(t) = \Lambda_2(t), \, t \in [0, 1] \quad \mbox{vs.} \quad \mathrm{H}_1: \|\Lambda_1 - \Lambda_2\|_{\infty} > r,
\end{equation}
where for a function $f:\mathbb{R} \rightarrow \mathbb{R}$ we denote $\|f\|_{\infty} = \sup_{t \in [0, 1]} |f(t)|$. To satisfy the distributed differential privacy constraint, we restrict the class of tests to
\begin{equation*}
    \Phi_{n_1, n_2, \epsilon_1, \epsilon_2, \delta_1, \delta_2} = \left\{\phi(M_1, M_2): M_k \text{ is $(\epsilon_k, \delta_k)$-differentially private w.r.t.~} D_k, \,\mathrm{Range}(\phi)=\{0, 1\} \right\}.
\end{equation*}

\subsection{Testing guarantees and optimality}

For the hypotheses in \eqref{eq-two-sample-test}, we construct tests of the form
\begin{equation} \label{eq-two-sample-test-def}
    \phi = \mathbbm{1}\left\{\|\widehat{\Lambda}_1 - \widehat{\Lambda}_2\|_{\infty} > \tau \right\},
\end{equation}
where $\tau>0$ is a threshold to be specified and $\widehat{\Lambda}_1(\cdot)$ and $\widehat{\Lambda}_2(\cdot)$ are outputs of \Cref{alg:DP-NA} in \Cref{app:DP-NA_alg}, applied to $(D_1, \epsilon_1, \delta_1)$ and $(D_2, \epsilon_2, \delta_2)$ respectively.  

It is shown in \cite{FDPCox} that the output of \Cref{alg:DP-NA} is minimax rate optimal, up to poly-logarithmic factors, for estimating $\{\Lambda(t)\}_{t \in [0, 1]}$ under sup-norm loss. \Cref{prop:cumulative_upper} below gives a sufficient condition on the separation rate under which $\phi$ has non-trivial testing power.  

\begin{proposition} \label[proposition]{prop:cumulative_upper}
Let
    \begin{equation} \label{eq:two_sample_threshold}
        \tau = c\left\{\frac{1}{\sqrt{n_1}} + \frac{\log_2(\min\{\sqrt{n_1}, n_1\epsilon_1\})^2 \log(1/\delta_1)}{n_1 \epsilon_1} + \frac{1}{\sqrt{n_2}} + \frac{\log_2(\min\{\sqrt{n_2}, n_2\epsilon_2\})^2 \log(1/\delta_2)}{n_2 \epsilon_2}\right\}     
    \end{equation}
    be the threshold in \eqref{eq-two-sample-test-def} for some absolute constant $c>0$.  

    We have that (i) the test $\phi$ is in $\Phi_{n_1, n_2, \epsilon_1, \epsilon_2, \delta_1, \delta_2}$; and (ii) under \Cref{assp:indep_atrisk}, there exists an absolute constant $C>0$ such that if the separation $r$ in \eqref{eq-two-sample-test} satisfies
    \begin{equation} \label{eq-two-sample-upper-rate}
    r \geq C \left(\frac{\mathrm{polylog}(n_1, \epsilon_1, \delta_1)}{\min\{n_1^{1/2}, n_1 \epsilon_1\}} +  \frac{\mathrm{polylog}(n_2, \epsilon_2, \delta_2)}{\min\{n_2^{1/2}, n_2 \epsilon_2\}}\right),
    \end{equation}
    then $\mathbb{P}_{\mathrm{H}_0}(\phi = 1) + \mathbb{P}_{\mathrm{H}_1} (\phi = 0) < 1/3$.
\end{proposition}

\Cref{prop:cumulative_lower} shows a matching lower bound up to poly-logarithmic factors. 

\begin{proposition} \label[proposition]{prop:cumulative_lower}
 The minimax separation rate of the test \eqref{eq-two-sample-test} is lower bounded by 
\begin{align*}
    r^*(n_1, n_2, \epsilon_1, \epsilon_2, \delta_1,  \delta_2) &= \inf \bigg\{ r>0;\, \exists \ \phi \in \Phi_{n_1, n_2, \epsilon_1, \epsilon_2, \delta_1, \delta_2} \\
    \mathrm{\,\,s.t.\,\,}   \forall &\mathbb{P}_{\Lambda_1}, \mathbb{P}_{\Lambda_2} \in \mathcal{C}:\sup_{t \in [0, 1]} |\Lambda_1(t) - \Lambda_2(t)| > r, \,\mathbb{P}_{\mathrm{H}_0}(\phi=1) + \mathbb{P}_{\mathrm{H}_1}(\phi=0) < \tfrac{1}{3}
     \bigg\}\\
     &\gtrsim \frac{1}{\min\{n_1^{1/2}, n_1 (\epsilon_1 + \delta_1)\}} + \frac{1}{\min\{n_2^{1/2}, n_2 (\epsilon_2 + \delta_2)\}}.
\end{align*}
\end{proposition}

Taken together,  Propositions~\ref{prop:cumulative_upper} and~\ref{prop:cumulative_lower} show that our minimax separation rate is optimal up to poly-logarithmic factors when $\delta_k \lesssim \epsilon_k$, $k \in \{1, 2\}$, which is the standard regime in differential privacy. 
The terms $n_1^{-1/2}$ and $n_2^{-1/2}$ correspond to the usual non-private testing rate, while $(n_1\epsilon_1)^{-1}$ and $(n_2\epsilon_2)^{-1}$ capture the cost of privacy. The proof of \Cref{prop:cumulative_lower} follows a general idea from \cite{arias2018remember}, where by fixing one sample as a reference distribution, the two-sample problem can be embedded into a one-sample testing problem for the cumulative hazard function; see \Cref{app:cumulative_lower} for details. 

Finally, the same construction also yields a nearly minimax optimal one-sample test by comparing $\widehat{\Lambda}_1(\cdot)$ against a pre-specified null cumulative hazard function $\Lambda_0(\cdot)$. For this problem, the rejection threshold may be calibrated by Monte Carlo simulations of under the null distribution.

\section{Numerical experiments}
\label{app:numerical}
In this section, we examine the numerical performance of our methods with simulated data. The code to reproduce these experiments can be found at \url{https://github.com/EKHung/DP_survival_tests}.  

\subsection{Testing Cox regression coefficient}
We generate failure time data according to \eqref{eq:hazard}, with $\lambda_0(t) = 1$ and covariates generated as $Z_{ij} \overset{\textrm{i.i.d.}}{\sim} \textrm{Uniform}(-1/\sqrt{d}, 1/\sqrt{d})$ for $i \in \{1, \dots, n\}$ and $j \in \{1, \dots, d\}$. The censoring times are generated from an Exp(0.3) distribution, and we take $T_i = \min\{\tilde{T}_i, C_i, 1\}$, $\Delta_i = \mathbbm{1}\{\tilde{T}_i \leq \min(C_i, 1)\}$. We vary the privacy budget $\epsilon \in \{1, 2, 3, 4\}$ and the number of samples $n \in \{3000, 3500, \dots, 6000\}$.  

In Figures~\ref{fig:binary_test} and \ref{fig:score_test}, we depict the empirical power and Type-I error of the test from \Cref{prop:beta_binary_upper} and \Cref{lemma:trace_plugin} with $\beta_0=0$, respectively. In both figures, we observe the trend of increasing power as $n$ and $\epsilon$ increase, echoing our theoretical findings.

\begin{figure}[htbp]
  \centering
\includegraphics[width=0.9\textwidth]{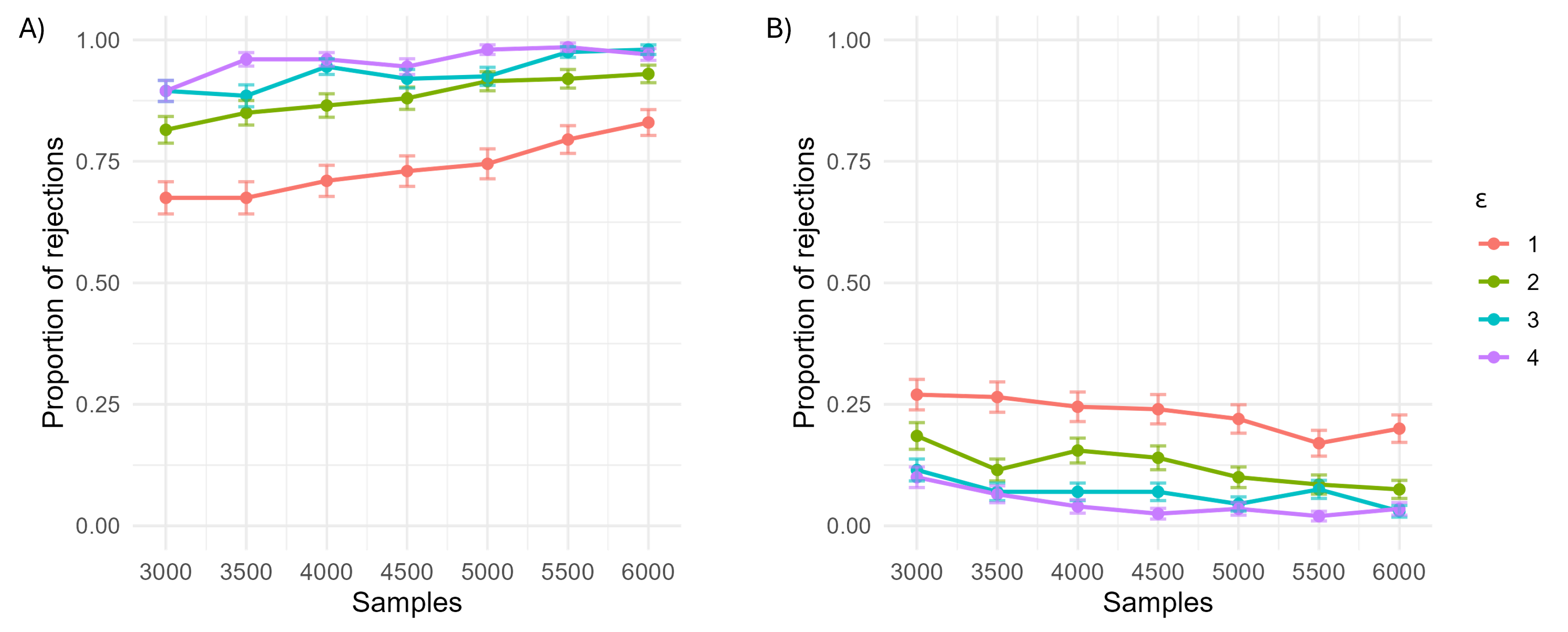}
  \caption{Proportion of rejections from applying the binary likelihood ratio test from \Cref{prop:beta_binary_upper} with $\beta_0=(0, 0, 0)$ and $\beta_1 = (0.2, 0.2, 0.2)$ to simulated data, over 200 repetitions, with bars showing standard error. The true parameter values are $\beta^*=(0.2, 0.2, 0.2)$ (panel A) and $\beta^*=(0, 0, 0)$ (panel B).}
  \label{fig:binary_test}
\end{figure}

\begin{figure}[htbp]
  \centering
\includegraphics[width=0.9\textwidth]{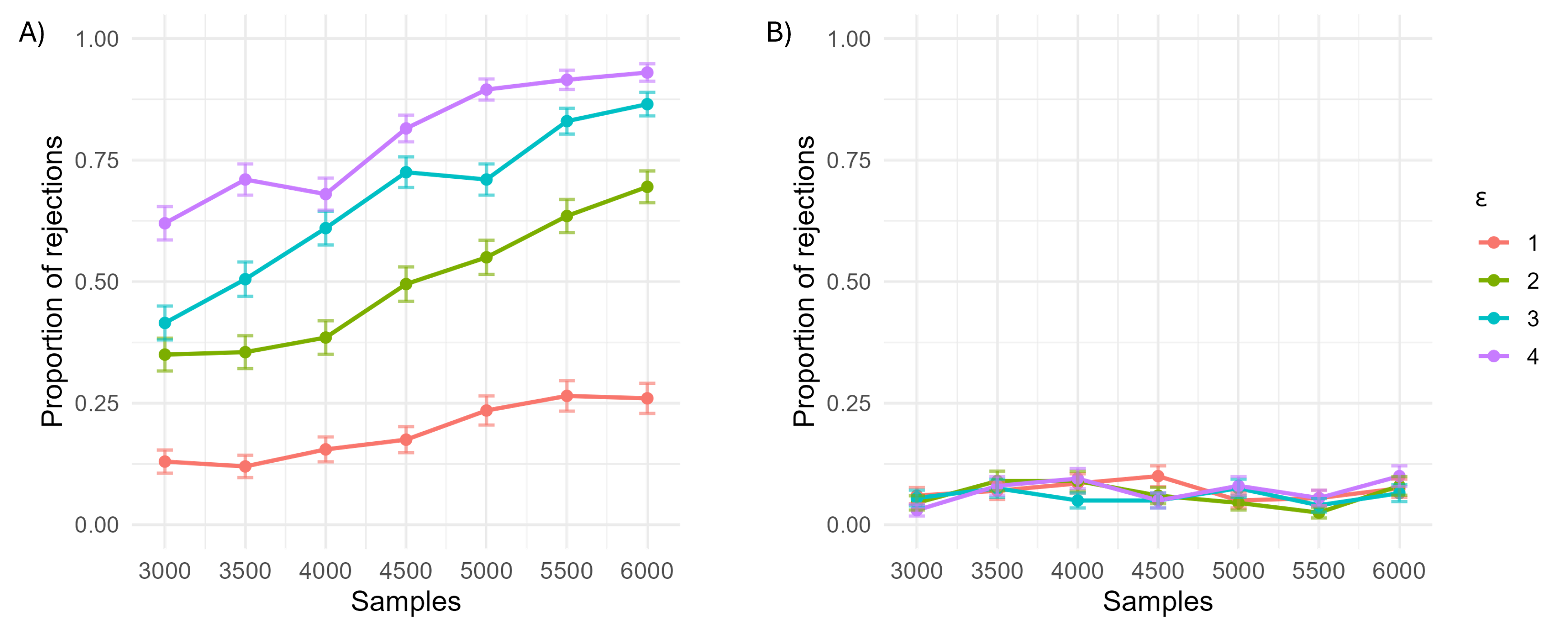}
  \caption{Proportion of rejections from applying the score test from \Cref{lemma:trace_plugin} to simulated data, over 200 repetitions, with bars showing standard error. The true parameter values are $\beta^*=(0.2, 0.2, 0.2)$ (panel A) and $\beta^*=(0, 0, 0)$ (panel B).}
  \label{fig:score_test}
\end{figure}

We take the covariate bound in \Cref{assp:covariates} to be $C_Z=1$, based on the true covariate generating process. The rejection threshold for the score test depends on two tuning parameters: $c_1$ and $c_2$ in \eqref{eq:beta_composite_threshold}, which are based on high probability upper bounds in the non-private test
under $\mathrm{H}_0$ and on the privacy-preserving noise. In our experiments, we set $c_1=0.5$ and $c_2=2$. While $c_2$ can be chosen according to sub-exponential tail bounds, $c_1$ depends on the data generating process; similar hyper-parameter tuning issues for tests have been discussed in \cite{kazan2023test}. A practitioner with further knowledge about the data generating process may be able to obtain a Monte Carlo threshold estimate for type-I error control based on the privatized statistic generated under the null distribution. We provide an example of this in \Cref{fig:bin_testq15} for the binary hypothesis test from \Cref{prop:beta_binary_upper}. All setups remain the same as those described previously, except that the rejection threshold is chosen to control the type-I error at level 0.15, using 2000 Monte Carlo samples of the privatized binary likelihood ratio test statistic under the null $\beta_0=(0, 0, 0)$. We see that the power improves as $n$ and/or $\epsilon$ increases, whilst keeping a desirable type-I error control, justifying our theoretical findings.

\begin{figure}[htbp]
  \centering
\includegraphics[width=0.9\textwidth]{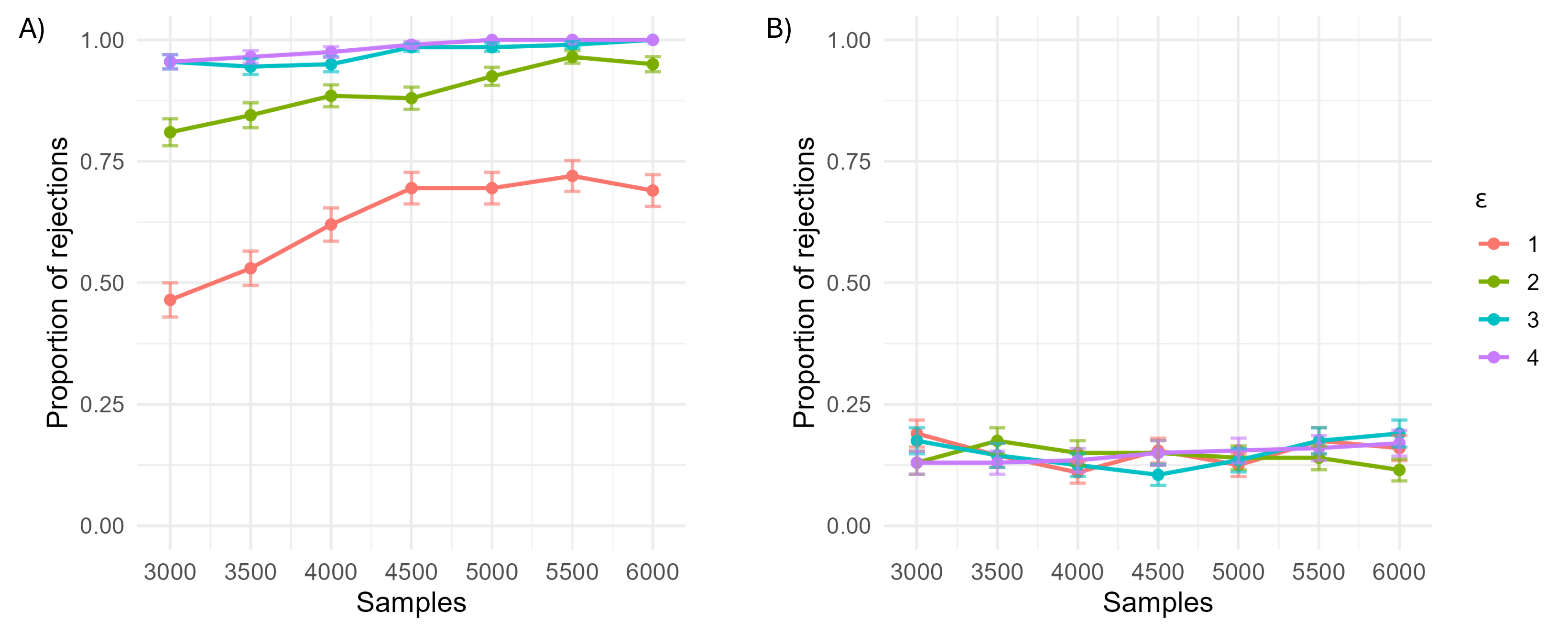}
  \caption{Proportion of rejections from applying the binary hypothesis test from \Cref{prop:beta_binary_upper} with an estimated rejection threshold sampled from the distribution with $\beta^*=(0, 0, 0)$, over 200 repetitions, with bars showing standard error. The true parameter values are $\beta^*=(0.2, 0.2, 0.2)$ (panel A) and $\beta^*=(0, 0, 0)$ (panel B).}
  \label{fig:bin_testq15}
\end{figure}

\subsection{Two-sample testing}
In this subsection, we investigate the numerical performance of the method in \Cref{prop:cumulative_upper}. We consider two servers with the same number of samples and privacy budgets. At the first server, we generate $\tilde{T}_{1, i} \overset{\text{i.i.d.}}{\sim} \mathrm{Exp}(1)$, $C_{1, i}\overset{\text{i.i.d.}}{\sim} \mathrm{Exp}(0.3)$ and we take $T_{1, i} = \min\{\tilde{T}_{1, i}, C_{1, i}\}$ and $\Delta_{1, i} = \mathbbm{1}\{\tilde{T}_{1, i} \leq C_{1, i}\}$. We repeat this procedure at the second server, but generate $\tilde{T}_{2, i} \overset{\text{i.i.d.}}{\sim} \mathrm{Exp}(1 + \gamma)$. We show how increasing $\gamma$ increases the proportion of rejections in  \Cref{fig:two_sample}(A), where there are $n_1=n_2=5000$ samples at each server. In \Cref{fig:two_sample}(B), we estimate the type-I error control by reporting the proportion of rejections when $\gamma=0$ at the second server, varying the number of samples as $n \in \{3000, \dots, 6000\}$. In each sub-figure of \Cref{fig:two_sample}, the privacy budgets are varied as $\epsilon\in \{1, 2, 3, 4\}$ and we fix $\delta = 0.001$ and set the tuning parameter in \eqref{eq:two_sample_threshold} to be $c=2$. We see an improvement in power as sample size, privacy budget or the separation between two hypotheses increases, while retaining reasonable type-I error control. 

\begin{figure}[htbp]
  \centering
\includegraphics[width=0.9\textwidth]{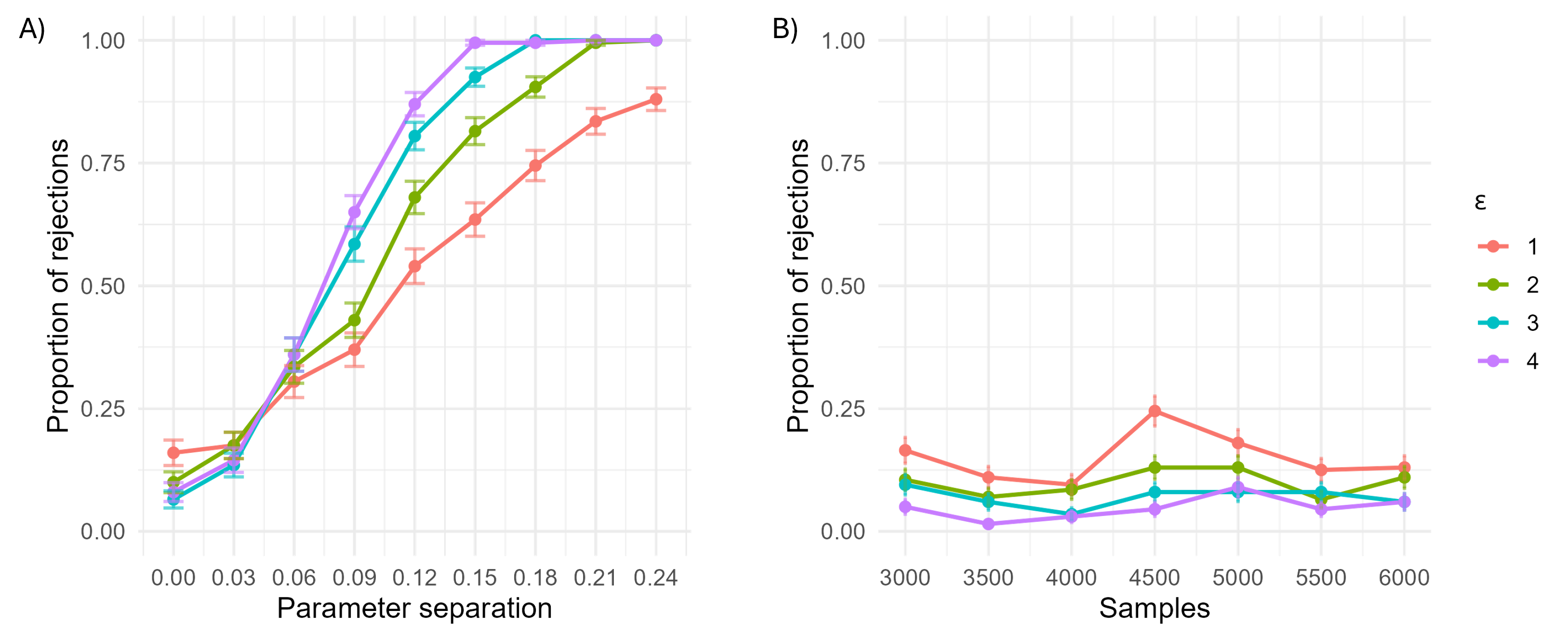}
\caption{Panel A shows the proportion of rejections from applying the two-sample test from \Cref{prop:cumulative_upper} to $n_1=n_2=5000$ simulated samples at each server, with bars showing standard error over 200 repetitions. In Panel B, the failure time distribution at both servers is $\mathrm{Exp}(1)$.}
  \label{fig:two_sample}
\end{figure}

\section*{Acknowledgements}
Hung is supported by the Chancellors' Scholarship scheme and the Statistics Centre for Doctoral Training at the University of Warwick. Yu is partially supported by the Philip Leverhulme Prize and EPSRC programme grant EP/Z531327/1.

\bibliographystyle{apalike}
\bibliography{ref}

\newpage
\appendix 
\section*{Appendices}
The proofs of results in \Cref{sec-2} and \Cref{sec-3} of the main text can be found in \Cref{app:sec-2} and \Cref{app:sec-3} respectively. Additional technical details and the proof of results from \Cref{sec-4} of the main text are contained in \Cref{app:sec-4}. \Cref{app:aux_lemmas} contains auxiliary lemmas used in our proofs. 

We introduce some additional notation. For $v \in \mathbb{R}^d$, we denote $v^{\otimes 0} = 1$, $v^{\otimes 1} = v$, and $v^{\otimes 2} = v v^\top$, and for a matrix $M$, we use $\|M\|$ to denote its $2,2$ operator norm. For a distribution $P$, we use $P^{\otimes n}$ to denote its $n$-fold product distribution. We use $c, C, c_1, c_2$ etc.~to denote absolute constants that may differ between different equations. For two sequences of positive numbers $\{a_n\}$ and $\{b_n\}$, we write $a_n \lesssim b_n$ if there exist an absolute constant $c>0$ such that $a_n \leq c b_n$, and we write $a_n \asymp b_n$ if $a_n \lesssim b_n$ and $b_n \lesssim a_n$. 

\section{Technical details from \texorpdfstring{\Cref{sec-2}}{}}
\label[appendix]{app:sec-2}
This appendix contains the technical details of \Cref{sec-2}: the proofs of \Cref{prop:beta_binary_upper} and \Cref{prop:beta_binary_lower} are in \Cref{sec-proof-prop-1} and \Cref{app:beta_bin_lower} respectively. 
\subsection{Proof of \texorpdfstring{\Cref{prop:beta_binary_upper}}{}} \label[appendix]{sec-proof-prop-1}
\begin{proof} 
\
\\
\textbf{Privacy guarantee:} we will show the privacy guarantee by bounding the sensitivity of the log likelihood ratio. It would then follow from \Cref{lemma:laplace_mechanism} and the post-processing property of differential privacy \citep[e.g.][Proposition 2.1]{dwork2014algorithmic} that the test $\phi$ is $(\epsilon, 0)$-differentially private.

Let  $\beta_0, \beta_1  \in \mathbb{B}_{C_\beta}(0) \subset \mathbb{R}^d$ and $D, D'$ be any two datasets that satisfy $\sum_{i=1}^n \mathbbm{1}\{(T_i, \Delta_i, Z_i) \neq (T_i', \Delta_i', Z_i')\} \leq 1$. Let $\ell_n(\beta_0; D)$ and $\ell_n(\beta_0; D')$ be \eqref{eq:Cox_partial_likelihood} evaluated at the two datasets. We have that
\begin{align*}
    &\left|\ell_n\left( \beta_0; D \right) -  \ell_n \left( \beta_0; D'\right)  - \{\ell_n(\beta_1; D) - \ell_n(\beta_1; D')\}\right| \\
    & \quad\leq \left\|\dot{\ell}_n \left(\theta \beta_0  + (1-\theta) \beta_1; D \right) - \dot{\ell}_n \left(\theta \beta_0  + (1-\theta) \beta_1; D' \right) \right\|_2 \|\beta_0 - \beta_1\|_2, \quad \theta \in [0 , 1]\\
    & \quad\leq c_{\beta_0, \beta_1} (1+ \log(n))\|\beta_0 - \beta_1\|_2, 
\end{align*}
recalling that $c_{\beta_0, \beta_1}= 4C_Z + \exp(2\max\{\|\beta_0\|_2, \|\beta_1\|_2\} C_Z)(2C_Z + C_Z^2)$. The first inequality is from the mean value theorem and the Cauchy--Schwarz inequality and the second inequality follows from bounding the sensitivity of the gradient by Lemma 9 of \cite{FDPCox}, and that $\|\theta \beta_0 + (1-\theta)\beta_1\|_2 \leq \max\{\|\beta_0\|_2, \|\beta_1\|_2\}$.

\bigskip
\noindent \textbf{Testing error:} we now consider bounding the non-private statistic under $\mathrm{H}_0: \beta^* = \beta_0$. By a Taylor expansion, we have that 
    \begin{align}
     \ell_n(\beta^*) - \ell_n(\beta_1) &= -\dot{\ell}_n(\beta^*)^\top (\beta_1 - \beta^*) - \frac{1}{2}(\beta_1 - \beta^*)^\top \underline{\ddot{\ell}_n([\beta_1, \beta^*])} (\beta_1-\beta^*)     \label{eq:binaryLRT_taylor}
\\
     &= - \dot{\ell}_n(\beta^*)^\top (\beta_1 - \beta^*) + \frac{1}{2}(\beta_1 - \beta^*)^\top nG(\beta^*) (\beta_1 - \beta^*) \notag\\
    &\quad + \frac{1}{2}(\beta_1 - \beta^*)^\top \left\{- \underline{\ddot{\ell}_n([\beta_1, \beta^*])} - nG(\beta^*) \right\}(\beta_1-\beta^*) \notag\\ 
    &= (I) + (II) + (III), \label{eq:binaryLRT_terms}
     \end{align}
where $G(\beta^*)$ is as defined in \Cref{assp:eigenvalue} and we denote 
\begin{equation*}
    \underline{\ddot{\ell}_n([\beta_1, \beta^*])} = \int_0^1  \ddot{\ell}_n \left((1-t)\beta^* + t\beta_1 \right) \dint t.
\end{equation*}
We will now show high probability bounds for each of the terms in \eqref{eq:binaryLRT_terms}. 

\bigskip
\noindent \textbf{Term $(I)$}: it suffices to show that there exists $c_1>0$ such that under the conditions in \Cref{prop:beta_binary_upper},  
\begin{equation}
\mathrm{Var}\left(\dot{\ell}_n(\beta^*)^\top(\beta_1-\beta^*) \right) \leq c_1\|\beta_1-\beta^*\|_2^2 \frac{n}{d},
\label{eq:score_var}
\end{equation}
since it would then follow from Chebyshev's inequality that 
\begin{equation}
    \mathbb{P}\left( |(I)| \geq 4\|\beta^*-\beta_1\|_2 \sqrt{\frac{c_1n}{d}}\right) \leq \frac{1}{16}.
    \label{eq:binary_martingale}
\end{equation}
By bounding cross-covariance using the Cauchy--Schwarz inequality and using $2ab \leq a^2 + b^2$, we have that 
\begin{equation}
    \mathrm{Cov}(\dot{\ell}_n(\beta^*)) \preceq 2\mathrm{Cov}\left(\sum_{i=1}^n \int_0^1 (Z_i(t) - \mu(t) ) \dint M_i(t)\right) + 2\mathrm{Cov}\left(\sum_{i=1}^n \int_0^1 (\bar{Z}(t) - \mu(t)) \dint M_i(t) \right), 
    \label{eqref:loglik-covariance}
\end{equation}
where $M_i(t)$ are independent, square-integrable martingales on $[0, 1]$ with predictable variation 
\begin{equation*}
    \langle M_i\rangle(t) = \int_0^t Y_i(s) \exp(\beta^{*\top}Z_i(s)) \dint \Lambda_0(s). 
\end{equation*} 
Since we also have that $\bar{Z}(t) - \mu(t)$ is predictable and bounded by \Cref{assp:covariates}, by Theorem II.3.1 of \citealp{andersen1993statistical}, we can write  
\begin{equation}   
\mathrm{Cov}\left(\sum_{i=1}^n \int_0^1 \left\{\bar{Z}(t) - \mu(t) \right\} \dint M_i(t) \right) = n\mathbb{E} \left[\int_0^1 \left\{\bar{Z}(t) - \mu(t)\right\}^{\otimes 2} Y(t) \exp(\beta_0^\top Z(t)) \dint \Lambda_0(t) \right].   \label{eq:sample_mean_conv1}
\end{equation}
 The maximum eigenvalue of \eqref{eq:sample_mean_conv1} is less than or equal to 
\begin{equation}
    n\mathbb{E} \left[ \sup_{t \in [0, 1]} \|\bar{Z}(t) - \mu(t)\|_2^4\right]^{1/2} \mathbb{E} \left[  \left(\int_0^1 Y(t) \exp(\beta_0^\top Z(t)) \dint \Lambda_0(t) \right)^2 \right]^{1/2}.
    \label{eq:sample_mean_conv2}
\end{equation}
We may bound $\mathbb{E} \left[  \left(\int_0^1 Y(t) \exp(\beta_0^\top Z(t)) \dint \Lambda_0(t) \right)^2 \right]^{1/2}$ by an absolute constant under Assumptions \ref{assp:covariates} and \ref{assp:indep_atrisk}. We also have for some absolute constant $c>0$ that
 \begin{align}
    \mathbb{E} \left[ \sup_{t \in [0, 1]} \|\bar{Z}(t) - \mu(t)\|_2^4\right] &\leq \int_0^{(2C_Z)^4} \mathbb{P} \left( \sup_{t \in [0, 1]} \|\bar{Z}(t) - \mu(t)\|_2^4 > x\right) \dint x  \notag\\
    &\leq \left(\frac{c \log(n)}{\sqrt{n}} \right)^4 + 4 \int_{c \log(n) / \sqrt{n}}^{2C_Z} y^3 \mathbb{P} \left( \sup_{t \in [0, 1]} \|\bar{Z}(t) - \mu(t)\|_2 > y \right) \dint y \notag\\
    &\lesssim \frac{\log(n)^4}{n^2} + \frac{1}{n^2},
    \label{eq:sample_mean_conv3}
\end{align}
where the first inequality is due to \Cref{assp:covariates} and the final inequality is  due to Lemma 20 of \cite{FDPCox}. Since the first term on the right-hand side of $\eqref{eqref:loglik-covariance}$ is $2nG(\beta^*)$, under the assumptions on $d$ in \Cref{prop:beta_binary_upper}, we have that
\begin{equation*}
    \lambda_{\max}\left( \mathrm{Cov}(\dot{\ell}_n(\beta^*)) \right) \lesssim \frac{n}{d},
\end{equation*}
which implies \eqref{eq:score_var}.

\bigskip
\noindent \textbf{Term $(II)$ in \eqref{eq:binaryLRT_terms}}: by \Cref{assp:eigenvalue}, we have that 
\begin{equation}
    (II) \geq  \frac{n \rho_-}{2d} \|\beta_1 - \beta^*\|^2_2.
    \label{eq:binary_mean}
\end{equation}

\bigskip
\noindent \textbf{Term $(III)$ in \eqref{eq:binaryLRT_terms}}: to bound $(III)$, we start with 
\begin{align*}
\left\|\underline{\ddot{\ell}_n([\beta_1, \beta^*])} - \ddot{\ell}_n(\beta^*)\right\| &= \left\| \int_0^1  \left\{\ddot{\ell}_n \left((1-t)\beta^* + t\beta_1 \right) - \ddot{\ell}_n(\beta^*)\right\} \dint t \right\| \notag \\
    &\leq \int_0^1 \left\| \ddot{\ell}_n \left((1-t)\beta^* + t\beta_1 \right) - \ddot{\ell}_n(\beta^*) \right\| \dint t.
\end{align*}
Due to Lemma 3.2 of \cite{huang2013oracle}, for all $t \in [0, 1]$, we can bound 
\begin{align*}
    \left\| \ddot{\ell}_n \left((1-t)\beta^* + t\beta_1 \right) - \ddot{\ell}_n(\beta^*) \right\|  \leq\max \left\{ \left|e^{2C_Z t\|\beta^* - \beta_1\|_2} - 1 \right|, \left|e^{-2 C_Z t\|\beta^* - \beta_1\|_2
    }- 1 \right| \right\} \left\| \ddot{\ell}_n(\beta^*) \right\|.
\end{align*}
We therefore have that 
\begin{align}
\left\|\underline{\ddot{\ell}_n([\beta_1, \beta^*])} -\ddot{\ell}_n(\beta^*)\right\| &\lesssim \left\|\beta_1-\beta^*\|_2\|\ddot{\ell}_n(\beta^*) \right\| \notag \\
&\lesssim \|\beta_1 - \beta^*\|_2 \left( \left\|\ddot{\ell}_n(\beta^*) + nG(\beta^*) \right\| + \left\|-n G(\beta^*) \right\| \right).
\label{eq:perturbed_Hessian}
\end{align}

Since 
\begin{equation*}
    |(III)| \leq \|\beta_1-\beta^*\|^2_2 \left( \left\| - \underline{\ddot{\ell}_n([\beta_1, \beta^*])} + \ddot{\ell}_n(\beta^*) \right\| + \left\|  \ddot{\ell}_n(\beta^*) + nG(\beta^*) \right\|\right),
\end{equation*} 
by combining \eqref{eq:perturbed_Hessian} with \Cref{lemma:hessian_convergence} and \Cref{assp:eigenvalue}, we have for some absolute constant $c_3>0$ that
\begin{equation}
    \mathbb{P} \left( \left|(III) \right| > c_3 \|\beta_1 - \beta^*\|_2^3
    \left\{  \log(dn)^2\sqrt{n} + \frac{n}{d} \right\} + c_3 \|\beta_1 - \beta^*\|^2_2 \log(dn)^2 \sqrt{n} \right) \lesssim \frac{1}{n}. 
    \label{eq:binary-hessian_deviation}
\end{equation}

\bigskip
\noindent \textbf{Privacy noise}: finally, we can bound the noise added to preserve privacy by
\begin{equation}
    \mathbb{P} \left( \left|\frac{c_{\beta_0, \beta_1} (1+\log(n))\|\beta^* - \beta_1\|_2}{\epsilon} W \right| > \frac{c_{\beta_0, \beta_1} \log(16) (1+\log(n))\|\beta^*-\beta_1\|_2}{\epsilon} \right) \leq \exp(-\log(16))
    \label{eq:binary-privacynoise}
\end{equation}
using the tail bound for a standard Laplace random variable. 

\bigskip
\noindent \textbf{Conclusion}: there exist absolute constants $C_1, C_2>0$ such that under the conditions in \Cref{prop:beta_binary_upper}, it holds that 
\begin{align}
    &4\sqrt{\frac{c_1n}{d}} \|\beta^* - \beta_1\|_2 + c_3\|\beta_1 - \beta^*\|_2^3
    \left\{  \log(dn)^2\sqrt{n} + \frac{n}{d} \right\} \notag\\
    & \quad+ c_3\|\beta_1 - \beta^*\|^2_2 \log(dn)^2 \sqrt{n} + \frac{c_{\beta_0, \beta_1}(1+\log(n))\|\beta^* - \beta_1\|_2}{\epsilon} \leq \frac{n\rho_-}{3d} \|\beta^* - \beta_1\|_2^2.
    \label{eq:beta_binary_sep}
\end{align}

For $\beta_0, \beta_1$ satisfying \eqref{eq:beta_binary_sep}, by a union bound argument with \eqref{eq:binary_martingale}, \eqref{eq:binary_mean}, \eqref{eq:binary-hessian_deviation} and \eqref{eq:binary-privacynoise}, we have that 
\begin{align*}
    \mathbb{P}_{\beta_0}(\gamma(\beta_0, \beta_1) < 0 ) &\leq \mathbb{P} \left( |(I)| + |(III)| + \left|\frac{c_{\beta_0, \beta_1} (1+\log(n))\|\beta^* - \beta_1\|_2}{\epsilon} W \right| > (II) \right)  \\
    &\leq \frac{1}{8} + \frac{c'}{n}.
\end{align*}
By similar arguments considering the negative of \eqref{eq:binaryLRT_taylor}, we have that 
\begin{equation*}
    \mathbb{P}_{\beta_1}(\gamma(\beta_0, \beta_1) \geq  0) \leq \frac{1}{8} + \frac{c'}{n}. 
\end{equation*}
Hence, the type-I and -II error control holds for $n$ large enough. 
\end{proof}

\subsection{Proof of \texorpdfstring{\Cref{prop:beta_binary_lower}}{}}
\label[appendix]{app:beta_bin_lower}
\begin{proof}
It suffices to show that there exist $\beta_0, \beta_1 \in \mathbb{R}^d$ and an absolute constant $c > 0$ such that 
\begin{equation*}
        \|\beta_0 - \beta_1\|_2^2 = c \left( \frac{d}{n} + \frac{d}{n^2 (\epsilon + \delta)^2} \right),
    \end{equation*}
and corresponding $\mathbb{P}_{\beta_0} \in \mathcal{P}(\beta_0)$ and $ \mathbb{P}_{\beta_1} \in \mathcal{P}(\beta_1)$ (defined in \Cref{sec:assumptions} of the main text) such that no $(\epsilon, \delta)$-differentially private mechanism $\phi$ satisfies $\mathbb{P}_{\beta_0}^{\otimes n}(\phi=1) + \mathbb{P}_{\beta_1}^{\otimes n}(\phi=0) < \tfrac{1}{3}$. Without loss of generality, we take $C_Z$ from \Cref{assp:covariates} to be $C_Z=1$ since other values would change the rate by an absolute constant factor independent of $n, d, \epsilon$, and $\delta$. 

Let $\beta_0 = (0, \dots, 0)^\top \in \mathbb{R}^d$ and $\beta_1 = (r/\sqrt{d}, \dots, r/\sqrt{d})^{\top} \in \mathbb{R}^d$, where 
\begin{equation}\label{eq-proof-prop-1-r}
    r^2 \asymp \frac{d}{n} + \frac{d}{n^2 (\epsilon + \delta)^2}.    
\end{equation}
We construct the distribution $(T, \Delta, Z) \sim \mathbb{P}_{\beta_0}$ as follows: let  $Z \sim \mathcal{U}[-1/\sqrt{d}, 1/\sqrt{d}]^{\otimes d}$ and $\tilde{T} \sim \textrm{Exp}(1)$. Let $C \sim 1+\textrm{Exp}(1)$ independently and set $T = \min(\tilde{T}, C)$, and $\Delta = \mathbbm{1}\{\tilde{T} \leq C\}$. Construct $\mathbb{P}_{\beta_1}$ similarly, but with $\tilde{T} \mid Z \sim \mathrm{Exp}(\exp(\beta_1 ^\top Z))$. It follows from the proof of Proposition 10 of \cite{FDPCox} that  $\mathbb{P}_{\beta_0} \in \mathcal{P}(\beta_0)$ and $ \mathbb{P}_{\beta_1} \in \mathcal{P}(\beta_1)$.

By the data processing inequality, we have that 
    \begin{align*}
        D_{\mathrm{KL}}(\mathbb{P}_{\beta_0} \parallel \mathbb{P}_{\beta_1}) &\leq \mathbb{E}_Z \mathbb{E}_{\tilde{T} \sim \mathrm{Exp}(1)}\left[ \log \frac{\exp(-\tilde{T})}{\exp(\beta_1^\top Z) \exp \left\{-\exp(\beta_1^\top Z) \tilde{T}\right\}}\right] \\
        &= \mathbb{E}_Z \left[ \exp(\beta_1^\top Z) - 1 - \beta_1^\top Z \right] \\
        &\leq \frac{1}{2} \mathbb{E}_Z \left[(\beta_1^\top Z)^2 + O(|(\beta_1^\top Z)^3|) \right] \\
        &\lesssim r^2/d + r^3/d^{3/2} \asymp r^2/d,
    \end{align*}
provided that $r < 1$. Suppose that $d/n >d/(n^2\epsilon^2)$. For any $(\epsilon, \delta)$-differentially private mechanism $\phi$, we have that  
\begin{equation*}
D_{\mathrm{TV}}(\phi \mathbb{P}_{\beta_0}^{\otimes n}, \phi\mathbb{P}_{\beta_1}^{\otimes n}) \leq D_{\mathrm{TV}}(\mathbb{P}_{\beta_0}^{\otimes n}, \mathbb{P}_{\beta_1}^{\otimes n}) \leq \sqrt{\frac{1}{2} n D_{\mathrm{KL}}(\mathbb{P}_{\beta_0}, \mathbb{P}_{\beta_1})} \lesssim 1,
\end{equation*}
where the first inequality is from the data processing inequality, the second is Pinsker's inequality (e.g.~Lemma 15.6 in \citealp{Wainwright2019}), and the last from the choice of $r$ in \eqref{eq-proof-prop-1-r}.
Otherwise, suppose that $d/(n^2\epsilon^2) \geq d/n$. By Lemma 6.1 of \cite{karwa2017finite} we have that 
\begin{equation*}
    D_\infty^{\tilde{\delta}} (\phi\mathbb{P}_{\beta_0}^{\otimes n}, \phi\mathbb{P}_{\beta_0}^{\otimes n}) \leq 6n \epsilon D_{\mathrm{TV}}(\mathbb{P}_{\beta_0}, \mathbb{P}_{\beta_1}), \quad \tilde{\delta }= \exp(6n \epsilon D_{\mathrm{TV}}(\mathbb{P}_{\beta_0}, \mathbb{P}_{\beta_1})) n\delta D_{\mathrm{TV}}(\mathbb{P}_{\beta_0}, \mathbb{P}_{\beta_1}),
\end{equation*}
where $D_\infty^\delta(P, Q)$ denotes the $\delta$-approximate max-divergence between two distributions $P, Q$ \citep[e.g.][Definition 3.6]{dwork2014algorithmic}. By Lemmas C.4 and C.5 of \cite{cai2024optimalfederatedlearningnonparametric} and Pinsker's inequality, we therefore have that 
\begin{align*}
    D_{\mathrm{TV}}(\phi \mathbb{P}_{\beta_0}^{\otimes n}, \phi\mathbb{P}_{\beta_1}^{\otimes n}) &\leq 2\tilde{\delta} +  \sqrt{3n \epsilon D_{\mathrm{TV}}(\mathbb{P}_{\beta_0}, \mathbb{P}_{\beta_1})\{\exp( 6n \epsilon D_{\mathrm{TV}}(\mathbb{P}_{\beta_0}, \mathbb{P}_{\beta_1})) - 1\}} \\
    &\leq 2\exp(6n\epsilon r/\sqrt{d}) \frac{nr \delta}{\sqrt{d}} + 6 \frac{n\epsilon r}{\sqrt{d}},
\end{align*}
where the last line holds if $n\epsilon r/\sqrt{d} < 0.2$, since $e^{6x} - 1 \leq 12x$ for all $x \leq 0.2$. By the choice of $r$ in \eqref{eq-proof-prop-1-r}, we can therefore bound $ D_{\mathrm{TV}}(\phi \mathbb{P}_{\beta_0}^{\otimes n}, \phi\mathbb{P}_{\beta_1}^{\otimes n}) \lesssim 1$. We conclude the proof by noting that the sum of type-I and II errors for any test $\phi$ is equal to $1- D_{\mathrm{TV}}(\phi \mathbb{P}_{\beta_0}^{\otimes n}, \phi \mathbb{P}_{\beta_1}^{\otimes n})$.
\end{proof}

\section{Technical details from \texorpdfstring{\Cref{sec-3}}{}}
\label[appendix]{app:sec-3}
The proof of the non-private lower bound result of \Cref{lemma:beta_composite_lower} can be found in \Cref{app:composite_lower_bound}. The remaining sub-sections provide the technical details for proving the upper bound: the proof of \Cref{lemma:scoretest_theoretical} is in \Cref{app:scoretest_theoretical}, the bound on the sensitivity of the private trace estimator is in \Cref{sec-trace-sensitivity}, and the proof of \Cref{lemma:trace_plugin} is in \Cref{app:trace_plugin}. 

\subsection{Proof of \texorpdfstring{\Cref{lemma:beta_composite_lower}}{}}
\label[appendix]{app:composite_lower_bound}
\begin{proof} 
We define the minimax separation gap for the identity testing problem as 
\begin{align*} \label{eq-separation-def-id}
& r^*(d, n) = \inf \Big\{ r > 0 \ : \ \forall \beta_0 \in \mathbb{B}_{C_\beta}(0) \subset \mathbb{R}^d, \exists  \text{ a test } \phi \text{ s.t.~}  \forall \mathbb{P}_{\beta_0} \in \mathcal{P}(\beta_0) \text{ and } \mathbb{P}_{\beta_1} \in \mathcal{P}(\beta_1): \\ 
& \hspace{5cm} \|\beta_0 - \beta_1\|_2>r, \,\mathbb{P}_{\beta_0}^{\otimes n} (\phi=1) + \mathbb{P}_{\beta_1}^{\otimes n}(\phi=0) < 1/3  \Big\}, 
\end{align*}
where $\mathcal{P}(\beta)$ is defined in \Cref{sec:assumptions} of the main text. 

\medskip
\noindent \textbf{Construction.} 
Consider the model given by $Z \sim \mathcal{U}[-1/\sqrt{d}, 1/\sqrt{d}]^{\otimes d}, \,T \mid Z \sim \mathrm{Exp}(\exp(Z^\top \beta))$, and $C \sim \mathrm{Exp}(1) + 1$ independently of $Z$ and $\tilde{T}$, and denote this distribution by $\mathbb{P}_\beta$. It follows from the proof of Proposition 10 in \cite{FDPCox} that for $\|\beta\|_2$ small enough, $\mathbb{P}_\beta$ satisfies Assumptions~\ref{assp:covariates}, \ref{assp:indep_atrisk} and \ref{assp:eigenvalue}, i.e.~it would belong to $\mathcal{P}(\beta^*)$.  For some $r>0$ to be specified later, let $\mathcal{P}_1(r) = \{\mathbb{P}_{(r/\sqrt{d}) \theta}, \,\theta \in \{\pm1\}^d\}$. We therefore have that
\begin{align*}
    r^*(d, n) \geq \sup\left\{r>0: \mathbb{P}_0^{\otimes  n}(\phi=1)+\sup_{\mathbb{P}_1 \in \mathcal{P}_1(r)} \mathbb{P}_1^{\otimes n}(\phi=0)\geq \frac{1}{3} \,\,\forall \text{ tests } \phi \right\}.
\end{align*}

\bigskip
\noindent \textbf{Lower bound via the TV distance.}  We can lower bound 
\begin{align*}
   \inf_{\phi} \sup_{\mathbb{P}_1 \in \mathcal{P}_1(r)} \left\{\mathbb{P}_0(\phi = 1) + \mathbb{P}_1 (\phi = 0)\right\} &\geq 1 - D_{\mathrm{TV}}(\mathbb{P}_0^{\otimes n}, \mathbb{P}_{\pi, n}) 
\end{align*}
where $\mathbb{P}_{\pi, n}$ is the mixture distribution of $\mathbb{P}_1$ with $\theta$ having i.i.d.~Radamacher coordinates. It now suffices to upper bound $D_{\mathrm{TV}}(\mathbb{P}_0^{\otimes n}, \mathbb{P}_{\pi,  n})$.

 Due to the data-processing inequality, we suppress the effect of the censoring in defining $\mathbb{P}_0^{\otimes n}$ for the null distribution and $\mathbb{P}_{\pi, n}$ for the mixture distribution below. We can bound 
\begin{align*}
    D_{\mathrm{TV}}(\mathbb{P}_0^{\otimes n}, \mathbb{P}_{\pi, n}) \leq \frac{1}{2} \sqrt{D_{\chi^2}(\mathbb{P}_0^{\otimes n},\mathbb{P}_{\pi, n})} &= \sqrt{\mathbb{E}_0\left[\left(\frac{\dint \mathbb{P}_{\pi, n}}{\dint \mathbb{P}_0^{\otimes n}}\right)^2\right] - 1}\\
    &=  \sqrt{\mathbb{E}_{\theta, \theta'}\left[\mathbb{E}_0 \left[ \frac{\dint\mathbb{P}_{(r/\sqrt{d})\theta}^{\otimes n}}{\dint\mathbb{P}_0^{\otimes n}}\frac{\dint\mathbb{P}_{(r/\sqrt{d})\theta'}^{\otimes n}}{\dint\mathbb{P}_0^{\otimes n}} \right] \right] - 1},
\end{align*}
where the inequality is by Cauchy--Schwarz and we can interchange the expectation from the mixture expansion due to Tonelli's theorem. 

For any $\theta, \theta'$, we have  
\begin{align*}
&\mathbb{E}_0 \left[ \frac{\dint\mathbb{P}_{(r/\sqrt{d})\theta}^{\otimes n}}{\dint\mathbb{P}_0^{\otimes n}}\frac{\dint\mathbb{P}_{(r/\sqrt{d})\theta'}^{\otimes n}}{\dint\mathbb{P}_0^{\otimes n}} \right] \\
&\quad=\mathbb{E}_{Z_{1:n}, T_{1:n} \sim \mathrm{Exp}(1)} \left[ \exp \left\{\left(\frac{r}{\sqrt{d}} \theta + \frac{r}{\sqrt{d}} \theta'\right)^\top \sum_{i=1}^n Z_i \right\} \exp \left\{ \sum_{i=1}^n -T_i \left(( e^{\frac{r}{\sqrt{d}} \theta^\top Z_i} + e^{\frac{r}{\sqrt{d}} \theta'^\top Z_i}) - 2 \right) \right\} 
\right] \\
&\quad= \prod_{i=1}^n \mathbb{E}_{Z_i} \left[\frac{ e^{\frac{r}{\sqrt{d}} \theta^\top Z_i + \frac{r}{\sqrt{d}} \theta'^\top Z_i}}{e^{\frac{r}{\sqrt{d}} \theta^\top Z_i} + e^{\frac{r}{\sqrt{d}} \theta'^\top Z_i} - 1} \right]\\
&\quad\stackrel{(a)}{=} \prod_{i=1}^n \mathbb{E}_{Z_i} \Bigg[ 1 + \left(\frac{r}{\sqrt{d}} \theta^\top Z_i\right)\left( \frac{r}{\sqrt{d}} \theta'^\top Z_i\right) + \frac{1}{2}\left(\frac{r}{\sqrt{d}} \theta^\top Z_i\right)^2\left(\frac{r}{\sqrt{d}} \theta'^\top Z_i\right)\\
& \hspace{4cm} +  \frac{1}{2} \left(\frac{r}{\sqrt{d}} \theta^\top Z_i \right) \left( \frac{r}{\sqrt{d}} \theta'^\top Z_i \right)^2 + O \left( \left|\frac{r}{\sqrt{d}} \theta^\top Z_i \right| + \left|\frac{r}{\sqrt{d}} \theta'^\top Z_i \right| \right)^4 \Bigg] \\
&\quad \stackrel{(b)}{=} \left(1 + c_1\frac{r^2}{d^2} \theta ^\top \theta' + O(r^4/d^2) \right)^n \\
&\quad \stackrel{(c)}{\leq} \exp \left\{ n \left(\frac{c_1r^2}{d^2} \theta ^\top \theta' + O(r^4/d^2) \right) \right\}
\end{align*}
where $(a)$ is by \Cref{lemma:lrt_taylor} and $(b)$ is due the data being i.i.d.~with
\begin{equation*}
    \mathbb{E}_Z\left[\left(\frac{r}{\sqrt{d}} \theta^\top Z \right)^2 \left(\frac{r}{\sqrt{d}} \theta'^\top Z \right) \right] = \frac{r^3}{d^{3/2}} \sum_{j, k, l=1}^ d\mathbb{E}_Z\left[\{ \theta\}_j \{ \theta\}_k \{\theta'\}_l \{Z\}_j\{Z\}_k\{Z\}_l \right] = 0,
\end{equation*}
where we use $\{x\}_k$ to denote the $k$-th coordinate of $x \in \mathbb{R}^d$ , and 
\begin{equation*}
    \mathbb{E}_Z\left[ \left(\frac{r}{\sqrt{d}} \theta^\top Z \right)^4 \right] = \frac{r^2}{d^2} \mathbb{E}_Z\left[\sum_{j=1}^d \{\theta\}_j^4 \{Z\}_j^4 +  6\sum_{j<k}^d \{\theta\}_j^2 \{Z\}_j^2 \{ \theta\}_k^2\{Z\}_k^2 \right] \asymp \frac{r^4}{d^2}
\end{equation*}
for $Z \sim \mathcal{U}[-1/\sqrt{d}, 1/\sqrt{d}]^{\otimes d}$. The inequality $(c)$ uses that $1+x \leq e^x$ for all $x \in \mathbb{R}$.  

By considering the m.g.f.~of Rademacher random variables, we have that  
\begin{align*}
    \mathbb{E}_{\theta, \theta'} \left[\exp \left(n \frac{c_1r^2}{d^2} \theta^\top \theta' \right) \right] 
    &= \cosh \left( \frac{nc_1r^2}{d^2} \right)^d \leq \exp\left( \frac{n^2c_1^2r^4}{2d^3} \right),
\end{align*}
where the final step is from the inequality $\cosh(x) \leq \exp(x^2/2)$ for all $x \in \mathbb{R}$. Therefore, by taking $r = c d^{3/4}/\sqrt{n}$ for some appropriate absolute constant $c>0$, we can bound $D_{\mathrm{TV}}(\mathbb{P}_0^{\otimes n}, \mathbb{P}_{\pi, n}) \leq 2/3$, which concludes the proof. 
\end{proof}

\subsection{Proof of \texorpdfstring{\Cref{lemma:scoretest_theoretical}}{}}
\label[appendix]{app:scoretest_theoretical}
\begin{proof}
The privacy guarantee follows from Lemma 9 in \cite{FDPCox}, which bounds the $\ell_2$-sensitivity of the partial likelihood gradient \eqref{eq:Cox_grad} for any $\beta_0 \in \mathbb{R}^d$ by 
\begin{equation*}
    \sup_{D, D': \sum_{i=1}^n \mathbbm{1}\{D_i \neq D_i'\} = 1} \left\| \frac{\dot{\ell}_n(\beta_0; D)}{\sqrt{n}} - \frac{\dot{\ell}_n(\beta_0; D')}{\sqrt{n}} \right\|_2 \leq \frac{C_{\beta_0} (1+\log(n))}{\sqrt{n}},
\end{equation*}
where we write $D = \{(T_i, \Delta_i, \{Z_i(t), t \in [0, 1]\})\}_{i=1}^n$ and similarly for $D'$. By the reverse triangle inequality, this is also a bound on the sensitivity of $f(D) = \|\dot{\ell}_n(\beta_0; D)/\sqrt{n}\|_2$. It follows from \Cref{lemma:laplace_mechanism} and the post-processing property of differential privacy \citep[e.g.~Proposition 2.1 in][]{dwork2014algorithmic} that the test $\phi$ is $(\epsilon, 0)$-differentially private.

For the error guarantee, we start by setting up some notation for the non-private part of the test. Using a Taylor expansion, we obtain
\begin{align*}
    \frac{\dot{\ell}_n(\beta_0)}{\sqrt{n}} &= \frac{\dot{\ell}_n(\beta^*)}{\sqrt{n}} +  \frac{1}{\sqrt{n}}\int_0^1 \ddot{\ell}_n(t\beta_0 + (1-t)\beta^*) \dint t (\beta_0 - \beta^*) \\
    &= A_n + B_n + C_n - \sqrt{n}G(\beta^*)(\beta_0-\beta^*),
\end{align*}
where we define 
\begin{equation*}
    A_n = \frac{1}{\sqrt{n}}\sum_{i=1}^n \int_0^1 (Z_i(t) - \mu(t)) \dint M_i(t), \quad B_n = \frac{1}{\sqrt{n}}\sum_{i=1}^n \int_0^1 (\mu(t) - \bar{Z}(t)) \dint M_i(t),
\end{equation*}
and 
\begin{equation*}
    C_n = \left\{\frac{1}{\sqrt{n}} \int_0^1 \ddot{\ell}_n(t\beta_0 + (1-t)\beta^*) \dint t + \sqrt{n} G(\beta^*)\right\} (\beta_0 - \beta^*).
\end{equation*}

\bigskip
\noindent \textbf{Type-I error:} by the triangle inequality, we have that
\begin{align}
    &\mathbb{P}\left(\|A_n + B_n\|_2 + \frac{C_{\beta_0} (1+\log(n))}{\sqrt{n} \epsilon}W > \sqrt{\mathrm{tr}\{G(\beta^*)\}} + \frac{c_1}{\sqrt{d}} + \frac{C_{\beta_0}(1+\log(n))\log(1/\alpha)}{\sqrt{n} \epsilon}\right) \notag \\
    &\quad\leq \mathbb{P}\left(\|A_n\|_2> \sqrt{\mathrm{tr}\{G(\beta^*)\}} + \frac{c_1}{2\sqrt{d}}\right) + \mathbb{P} \left( \|B_n\|_2 > \frac{c_1}{2 \sqrt{d}} \right) + \mathbb{P}\left( W >  \log(1/\alpha) \right)
    \label{eq:beta_binary_typeI}
\end{align}
We can bound the first term by 
\begin{align}    &\mathbb{P}\left(\|A_n\|_2 > \sqrt{\mathrm{tr}\{G(\beta^*)\}} + \frac{c_1}{2\sqrt{d}}\right) \notag\\
    &\quad= \mathbb{P}\left(\|A_n\|^2_2 > \mathrm{tr}\{G(\beta^*)\} + \sqrt{\frac{c_1 \mathrm{tr}\{G(\beta^*)\}}{d}} + \frac{c_1^2}{4d}\right) \notag\\
    &\quad\leq \mathbb{P}\left( \left|\|A_n\|^2_2 - \mathbb{E}[\|A_n\|^2_2] \right| \geq \mathrm{tr}\{G(\beta^*)\} + \sqrt{\frac{c_1 \mathrm{tr}\{G(\beta^*)\}}{d}} + \frac{c_1^2}{4d} - \mathbb{E}[\|A_n\|^2_2] \right) \notag\\
     &\quad\leq \mathbb{P}\left( \left|\|A_n\|^2_2 - \mathbb{E}\left[\|A_n\|^2_2 \right] \right| \geq  \sqrt{\frac{c_1 \mathrm{tr}\{G(\beta^*)\}}{d}} + \frac{c_1^2}{4d} \right)
     \label{eq:scoretest_type1}\\
     &\quad\leq \frac{4\mathrm{tr}\{G(\beta^*)^2\} + c/n}{(\sqrt{c_1 \mathrm{tr}\{G(\beta^*)\}/d} + c_1^2/(4d))^2} \lesssim \frac{1}{c_1} \notag
\end{align}
where in the final line we used \Cref{assp:eigenvalue}, \Cref{lemma:vector_expvar} and Chebyshev's inequality and the final inequality holds for $n, d$ large enough.   

To bound the second term of \eqref{eq:beta_binary_typeI}, note that by \eqref{eq:sample_mean_conv1}, \eqref{eq:sample_mean_conv2} and \eqref{eq:sample_mean_conv3}, we have that $\mathrm{Cov}(B_n)$ has maximum eigenvalue upper bounded by $c \log(n)^2/n$. It follows from Markov's inequality that 
\begin{equation}
    \mathbb{P} \left( \|B_n\|_2 \geq \frac{c_1}{2\sqrt{d}}  \right) \leq \frac{\mathbb{E}[\|B_n\|^2_2]}{c_1^2/4d} \leq \frac{c \log(n)^2d/n}{c_1^2/4d} \leq cC_1/c_1^2,
    \label{eq:sample_mean_convergence}
\end{equation}
where the final inequality follows from the conditions in \Cref{lemma:scoretest_theoretical} and is assumed to be sufficiently small for the error guarantee.

 By a tail bound for a standard Laplace random variable, the third term of \eqref{eq:beta_binary_typeI} can be bound by $\alpha/2$. We may scale $c_1$ based on the absolute constant $\rho_+$ from \Cref{assp:eigenvalue} so that by a union bound argument, for large enough $n, d$, we can bound $\mathbb{P}_{\mathrm{H}_0}(\phi=1) \leq 1/6$.  

\bigskip
\noindent \textbf{Type-II error:} we may bound 
\begin{align*}
    &\mathbb{P} \Bigg( \left\|A_n + B_n + C_n - \sqrt{n}G(\beta^*)(\beta_0-\beta^*) \right\|_2 + \frac{C_{\beta_0}(1+\log(n))}{\sqrt{n} \epsilon}W \\
    &\hspace{6cm}\leq \sqrt{\mathrm{tr}\{G(\beta^*)\}} + \frac{c_1}{\sqrt{d}} + \frac{C_{\beta_0}(1+\log(n)) \log(1/\alpha)}{\sqrt{n} \epsilon} \Bigg) \\
    &\quad \leq \mathbb{P} \Bigg( \left\|A_n + B_n + C_n - \sqrt{n}G(\beta^*)(\beta_0 - \beta^*) \right\|_2 \\
    &\hspace{6cm} \leq \sqrt{\mathrm{tr}\{G(\beta^*)\}} + \frac{c_1}{\sqrt{d}} + \frac{2C_{\beta_0}(1+\log(n)) \log(1/\alpha)}{\sqrt{n} \epsilon} \Bigg) + \frac{\alpha}{2}.  
\end{align*}
To simplify notation, we denote 
\begin{equation}
    \kappa = 2C_{\beta_0} (1+\log(n))\log(1/\alpha).
    \label{eq:kappa_notation}
\end{equation}
We can bound the first term by
\begin{align}
    &\mathbb{P}\left(\left\|A_n + B_n + C_n - \sqrt{n}G(\beta^*)(\beta_0-\beta^*) \right\|_2 \leq \sqrt{\mathrm{tr}\{G(\beta^*)\}} + \frac{c_1}{\sqrt{d}} + \frac{\kappa}{\sqrt{n} \epsilon}\right) \notag\\
    &\quad \leq \mathbb{P}\left( \left|\|A_n + B_n - \sqrt{n} G(\beta^*)(\beta_0 - \beta^*)\|_2 - \|C_n\|_2 \right| \leq  \sqrt{\mathrm{tr}\{G(\beta^*)\}} + \frac{c_1}{\sqrt{d}} + \frac{\kappa}{\sqrt{n} \epsilon} \right) \notag\\
    & \quad= \mathbb{P}\Bigg(\|C_n\|_2  - \sqrt{\mathrm{tr}\{G(\beta^*)\}} - \frac{c_1}{\sqrt{d}} - \frac{\kappa}{\sqrt{n} \epsilon}  \notag \\
    &\hspace{3cm} \leq \|A_n + B_n - \sqrt{n} G(\beta^*)(\beta_0 - \beta^*)\|_2 \leq \|C_n\|_2 + \sqrt{\mathrm{tr}\{G(\beta^*)\}} + \frac{c_1}{\sqrt{d}} + \frac{\kappa}{\sqrt{n} \epsilon} \Bigg) \notag\\
    & \quad\leq \mathbb{P} \left(\|A_n + B_n - \sqrt{n} G(\beta^*)(\beta_0 - \beta^*)\|_2 \leq  \sqrt{\mathrm{tr}\{G(\beta^*)\}} + \frac{c_1}{\sqrt{d}} + \frac{\kappa}{\sqrt{n} \epsilon}+ c_3 \left(\frac{r^2 \sqrt{n}}{d} + r \log(nd)^2 \right) \right) \notag\\
    &\quad \quad +  \mathbb{P} \left(\|C_n\|_2 \geq c_3 \left(\frac{r^2 \sqrt{n}}{d} + r \log(nd)^2 \right)\right).
    \label{eq:composite_hessian}
\end{align}
where we denote $r= \|\beta_0 - \beta^*\|_2$. The inequality in \eqref{eq:composite_hessian} is from noting that for positive $X, Y$, for any $t, a >0$, we have that 
\begin{equation*}
    \{X-t \leq Y \leq X+t \} \subseteq \{Y \leq t+a\} \cup \{X \geq a\}.
\end{equation*}
We may bound the second term of \eqref{eq:composite_hessian} by $O(1/n)$, since 
\begin{align*}
    \|C_n\|_2 &\leq  \left\|\frac{1}{\sqrt{n}} \int_0^1 \ddot{\ell}_n(t\beta_0 + (1-t)\beta^*) \dint t + \sqrt{n} G(\beta^*) \right\| \|\beta_0 - \beta^*\|_2 \\
    &\leq \left( \frac{1}{\sqrt{n}} \int_0^1 \left\| \ddot{\ell}_n(t\beta_0 - (1-t) \beta^*) - \ddot{\ell}_n(\beta^*) \right\| \dint t + \left\| \frac{1}{\sqrt{n}} \ddot{\ell}_n(\beta^*) + \sqrt{n} G(\beta^*)\right\| \right) \|\beta_0 - \beta^*\|_2\\
    &\lesssim \frac{r^2}{\sqrt{n}}\|\ddot{\ell}_n(\beta^*)\| + r\left\| \frac{1}{\sqrt{n}} \ddot{\ell}_n(\beta^*) + \sqrt{n} G(\beta^*)\right\|
\end{align*}
where the final inequality is by similar arguments leading to \eqref{eq:perturbed_Hessian}. We have from \Cref{lemma:hessian_convergence} that we can bound 
\begin{equation*}
\left\| \frac{1}{\sqrt{n}} \ddot{\ell}_n(\beta^*) + \sqrt{n} G(\beta^*)\right\|\lesssim \log(nd)^2, 
\end{equation*}
with probability greater than $1-c/n$, which also implies that $\| \ddot{\ell}_n(\beta^*)\|_2 \lesssim n/d + \log(nd)^2\sqrt{n}$ with high probability, by a triangle inequality.  

As for the first term of \eqref{eq:composite_hessian}, we have that 
\begin{align}
     &\mathbb{P} \left(\|A_n + B_n - \sqrt{n} G(\beta^*)(\beta_0 - \beta^*)\|_2 \leq \sqrt{\mathrm{tr}\{G(\beta^*)\}} + \frac{c_1}{\sqrt{d}} + \frac{\kappa}{\sqrt{n} \epsilon}+ c_3 \left(\frac{r^2 \sqrt{n}}{d} + r \log(nd)^2 \right) \right) \notag\\
    & \quad\leq \mathbb{P} \left(\|A_n - \sqrt{n} G(\beta^*)(\beta_0 - \beta^*)\|_2 \leq \sqrt{\mathrm{tr}\{G(\beta^*)\}} + \frac{c_1}{\sqrt{d}} + \frac{\kappa}{\sqrt{n} \epsilon}+ c_3 \left(\frac{r^2 \sqrt{n}}{d} + r \log(nd)^2 \right) + c_4 \sqrt{\frac{d}{n}} \right) \notag\\
    &\quad \quad+ \mathbb{P}\left(\|B_n\|_2 \geq c_4 \sqrt{\frac{d}{n}} \right). \label{eq:beta_composite_alternative2}
\end{align}
The second probability in \eqref{eq:beta_composite_alternative2} can be bound using \eqref{eq:sample_mean_convergence} by e.g.~1/20, under the conditions in \Cref{lemma:scoretest_theoretical} with appropriate absolute constants. 
Letting $v= -\sqrt{n}G(\beta^*)(\beta_0 - \beta^*)$ and $q= c_5/\sqrt{d} + \kappa/(\sqrt{n} \epsilon)+ c_3\left(r^2 \sqrt{n} /d + r \log(nd)^2 \right)$, where $c_5 = \max\{c_1 + c_4\sqrt{C_1}, 10\rho_+$\}, we can bound the first probability of \eqref{eq:beta_composite_alternative2} by 
\begin{align}
    &\mathbb{P}\left(\|A_n+v\|^2_2 \leq \left(\sqrt{\mathrm{tr}\{G(\beta^*)\}} +q \right)^2 \right) \leq \mathbb{P}\left(\|A_n\|^2_2 \leq \mathrm{tr}\{G(\beta^*)\} - \frac{c_5}{\sqrt{d}} \right) \notag\\
    &\hspace{4cm}+ \mathbb{P} \left( 2A_n^\top v \leq \left(\sqrt{\mathrm{tr}\{G(\beta^*)\}} + q \right)^2 - \|v\|_2^2 -  \mathrm{tr}\{G(\beta^*)\} + \frac{c_5}{\sqrt{d}} \right).
    \label{eq:beta_comp_shifted}
\end{align}
The first term of \eqref{eq:beta_comp_shifted} can be bounded by
\begin{align*}
\mathbb{P} \left(  \left| \|A_n\|^2_2 - \mathrm{tr}\{G(\beta^*)\} \right| \geq \frac{c_5}{\sqrt{d}} \right) &\leq \frac{\mathrm{Var}(\|A_n\|_2^2)}{c_5^2/d} =  \frac{4 \mathrm{tr}\{\mathrm{Cov}(A_n)^2\} + c/n}{c_5^2/d} \\
&\leq \frac{4 \rho_+^2/d + c/n}{c_5^2/d} \leq \frac{1}{20} 
\end{align*}
where the first line is by Chebyshev's inequality and \Cref{lemma:vector_expvar} and the second line holds by \Cref{assp:eigenvalue} and $n$ large enough. 

Under \Cref{assp:eigenvalue} and large enough absolute constants $C_1, C_2$ in the conditions of \Cref{lemma:scoretest_theoretical}, we can bound 
\begin{align*}
\sqrt{\mathrm{tr}\{G(\beta^*)\}} \left(\frac{1}{\sqrt{d}} + \frac{\kappa}{\sqrt{n} \epsilon}+ \frac{r^2 \sqrt{n}}{d} + r \log(nd)^2 \right) +  \frac{c_5}{\sqrt{d}} + \frac{\kappa^2}{n \epsilon^2} + \left(\frac{r^2\sqrt{n}}{d}  + r \log(nd)^2 \right)^2  \lesssim \frac{nr^2}{d^2},
\end{align*}
which leads to  
\begin{align*}
    \left(q - \sqrt{\mathrm{tr}\{G(\beta^*)\}} \right)^2 - \|v\|^2_2 - \mathrm{tr}\{G(\beta^*)\} + \frac{c_5}{\sqrt{d}} = 2q\sqrt{\mathrm{tr}\{G(\beta^*)\}} + q^2 + \frac{c_5}{\sqrt{d}} -\|v\|^2_2 \asymp -\frac{nr^2}{d^2}.
\end{align*}
Since $A_n^\top v$ is mean zero and  $\mathrm{Var}(A_n^\top v) = v^\top G(\beta^*) v \asymp nr^2/d^3$, by Chebyshev's inequality, we can upper bound the second term of \eqref{eq:beta_comp_shifted} by 
\begin{equation*}
    \mathbb{P} \left( |2A_n^\top v| \geq \frac{nr^2}{d^2}  \right) \leq \frac{4 \mathrm{Var}(A_n^\top v)}{(nr^2/d^2)^2} \lesssim \frac{d}{nr^2} \lesssim \frac{1}{\sqrt{d}},
\end{equation*}
where the last inequality holds under \eqref{eq:beta_composite_assumption}. We can therefore bound the type-II error probability by
\begin{align*}
    \mathbb{P}_{\mathrm{H}_1} \left( \phi=0 \right) \leq \frac{c}{\sqrt{n}} + \frac{c}{\sqrt{d}} + \frac{1}{10} + \frac{\alpha}{2} \leq \frac{1}{6}
\end{align*}
where the second inequality holds for suitably chosen $\alpha$ and $n, d$ large enough.
\end{proof}

\subsection{Sensitivity of the trace estimator} \label[appendix]{sec-trace-sensitivity}

\begin{lemma} \label[lemma]{lemma:trace_sensitivity}
Let $H(\,\cdot\, ; \beta_0)$ be as defined in \eqref{eq:truncated_Hessian}. For any two size $n$ datasets $D, D'$ that differ in one entry, we have that 
\begin{align}
    & \left|\mathrm{tr}\{H(D; \beta_0)\} - \mathrm{tr}\{H(D'; \beta_0)\} \right| \notag \\
    &\leq\frac{C_Z^2}{n}
\Bigg\{2 +\exp\!\left(2C_Z\|\beta_0\|_2\right)(6+4\log (n))
+2\exp\!\left(4C_Z\|\beta_0\|_2\right) \notag \\
&\hspace{1.5cm}+\frac{\exp\!\left(3C_Z\|\beta_0\|_2\right)(1+\log (n))
+6\exp\!\left(2C_Z\|\beta_0\|_2\right)}{n}
+\frac{2\exp\!\left(4C_Z\|\beta_0\|_2\right)(1+\log (n))}{n^2}
\Bigg\}.
    \label{eq:trace_sensitivity}
\end{align}   
The right-hand side of \eqref{eq:trace_sensitivity} therefore serves as a bound for the sensitivity of the trace of $H(\,\cdot\, ; \beta_0)$ and we denote this quantity by $K(n, \beta_0)$ in the main text. 
\end{lemma}

\begin{proof}
    We can rewrite  
\begin{equation*}
 H(D; \beta_0)  = \frac{1}{n} \sum_{i=1}^n \int_0^1 \left(\frac{S^{(2)}(D; t, \beta_0)}{{S}^{(0)}(D; t, \beta_0)} - \left\{ \frac{S^{(1)}(D; t, \beta_0)}{S^{(0)}(D; t, \beta_0)} \right\}^{\otimes 2}\right)  \dint N_i(t),
\end{equation*}
in which we denote
\begin{equation*}
    S^{(k)}(D; t, \beta_0) = \frac{1}{n}\sum_{i=1}^n Y_i(t) \exp(\beta_0^\top Z_i(t)) Z_i(t) ^{\otimes k}, \quad k \in \{0, 1, 2\}.
\end{equation*}

Without loss of generality, assume that the different entry is in the first observation, so that $D=(T_1, \Delta_1, Z_1) \cup \{(T_i, \Delta_i, Z_i)\}_{i=2}^n$ and $D'=(T_1', \Delta_1', Z_1') \cup \{(T_i, \Delta_i, Z_i)\}_{i=2}^n$. By the triangle inequality we have that 
\begin{align}
    &|\mathrm{tr}\{H(D; \beta_0)\} - \mathrm{tr}\{H(D'; \beta_0)\}| \notag\\
    &\leq \frac{1}{n} \int_0^1 \left| \mathrm{tr} \left(\frac{S^{(2)}(D; t, \beta_0)}{S^{(0)}(D; t, \beta_0)} - \left\{ \frac{S^{(1)}(D; t, \beta_0)}{S^{(0)}(D; t, \beta_0)} \right\}^{\otimes 2} \right) \right| \dint N_1(t)  \notag\\
    &\quad + \frac{1}{n}  \int_0^1 \left| \mathrm{tr} \left(\frac{S^{(2)}(D'; t, \beta_0)}{S^{(0)}(D'; t, \beta_0)} - \left\{ \frac{S^{(1)}(D'; t, \beta_0)}{S^{(0)}(D'; t, \beta_0)} \right\}^{\otimes 2} \right) \right|\dint N_1'(t) \notag \\ 
    &\quad + \frac{1}{n} \sum_{i=2}^n \int_0^1  \left| \mathrm{tr} \left(\frac{S^{(2)}(D; t, \beta_0)}{S^{(0)}(D; t, \beta_0)}  - \frac{S^{(2)}(D'; t, \beta_0)}{S^{(0)}(D'; t, \beta_0)} \right)\right| \dint N_i(t) \notag\\
    &\quad + \frac{1}{n} \sum_{i=2}^n \int_0^1 \left| \mathrm{tr} \left(\frac{S^{(1)}(D; t, \beta_0)^{\otimes 2}}{S^{(0)}(D; t, \beta_0)^2} - \frac{S^{(1)}(D'; t, \beta_0)^{\otimes 2}}{S^{(0)}(D'; t, \beta_0)^2} \right)\right| \dint N_i(t) \notag\\ 
    &= (I) + (II) + (III) + (IV).     \label{eq:trace_sensitivity_triangle}
\end{align}

\medskip
\noindent \textbf{Term (I) and (II):} we can bound each of $(I)$ and $(II)$ in \eqref{eq:trace_sensitivity_triangle} by $C_Z^2/n$. This is because by \Cref{assp:covariates}, we have that as weighted averages, for all $t \in [0, 1]$
\begin{equation}
    \mathrm{tr} \left(\frac{S^{(2)}(D; t, \beta_0)}{S^{(0)}(D; t, \beta_0)} \right) =  \frac{\sum_{i=1}^n Y_i(t) \exp(Z_i(t)^\top \beta_0)) \|Z_i\|_2^2}{\sum_{i=1}^n Y_i(t) \exp(Z_i(t)^\top \beta_0)}   \in [0, C_Z^2]
    \label{eq:deg2_bound}
\end{equation}
and similarly 
\begin{equation}
     \mathrm{tr} \left( \left\{\frac{S^{(1)}(D; t, \beta_0)}{S^{(1)}(D; t, \beta_0)} \right\}^{\otimes 2} \right) = 
     \left\|\frac{S^{(1)}(D; t, \beta_0)}{S^{(0)}(D; t, \beta_0)}\right\|_2^2 \in [0, C_Z^2].
     \label{eq:deg1_bound}
\end{equation}

\medskip
\noindent \textbf{Term (III):} we can write the integrand of $(III)$ as 
\begin{align*}
    & \left|\mathrm{tr} \left\{\frac{S^{(2)}(D; t, \beta_0)}{S^{(0)}(D; t, \beta_0)} - \frac{S^{(2)}(D'; t, \beta_0)}{S^{(0)}(D'; t, \beta_0)} \right\} \right| \\
    & = \left|\mathrm{tr} \left\{\frac{S^{(2)}(D; t, \beta_0)S^{(0)}(D'; t, \beta_0) - S^{(2)}(D'; t, \beta_0) S^{(0)}(D; t, \beta_0)}{S^{(0)}(D; t, \beta_0)S^{(0)}(D'; t, \beta_0)} \right\} \right| \\
    &\leq \frac{ |\mathrm{tr} \{S^{(2)}(D; t, \beta_0) - S^{(2)}(D'; t, \beta_0)\}|}{S^{(0)}(D; t, \beta_0)} + \frac{ \left|\mathrm{tr} \left\{S^{(2)}(D'; t, \beta_0) \left(S^{(0)}(D; t, \beta_0) - S^{(0)}(D'; t, \beta_0)\right) \right\} \right|}{S^{(0)}(D'; t, \beta_0) S^{(0)}(D;t, \beta_0)} \\
    &=a_1(t) + a_2(t).
\end{align*}

Integrating $a_1(t)$ against the jump process, we have 
\begin{align}
    \frac{1}{n} \sum_{i=1
    }^n \int_0^1 a_1(t) \dint N_i(t) &= \frac{1}{n} \sum_{i=1
    }^n \int_0^1 \frac{ |\mathrm{tr} \{Y_1(t) \exp(\beta_0^\top Z_1(t))Z_1^{\otimes 2} - Y_1'(t) \exp(\beta_0^\top Z_1'(t)) Z_1'^{\otimes 2}\}|}{\sum_{i=1}^n Y_i(t) \exp(\beta_0^\top Z_i(t))} \dint N_i(t) \notag \\
    &\leq \frac{\exp(2C_Z \|\beta_0\|_2)C_Z^2}{n} \int_0^1 \frac{1}{\sum_{j=1}^n Y_i(t)} \dint N_i(t) \notag \\
    &\leq \frac{\exp(2C_Z\|\beta_0\|_2)C_Z^2(1 + \log (n))}{n}.
    \label{eq:sens_log_example}
\end{align}
The final inequality is by noting that there are at most $n$ possible jump times, and at each of these jump times, $\sum_{j=1}^n Y_j(t)$ takes a distinct value in $\{1, \dots, n\}$ as the failure times are continuous random variables, and we may bound 
\begin{equation*}
    \sum_{i=1}^n \frac{1}{i} \leq 1 + \int_1^{n} \frac{1}{x} \dint x \leq 1 + \log(n).
\end{equation*}

We next bound $a_2(t)$ by 
\begin{align*}
    c_2(t) &\leq C_Z^2 \left| 1 - \frac{Y_1'(t) \exp(\beta_0^\top Z_1'(t)) + \sum_{j=2}^n Y_j(t) \exp(\beta_0^\top Z_j(t))}{Y_1(t) \exp(\beta_0^\top Z_1(t)) + \sum_{j=2}^n Y_j(t) \exp(\beta_0^\top Z_j(t))} \right| \\
    &\leq \frac{C_Z^2 |Y_1(t) \exp(\beta_0^\top Z_1(t)) - Y_1'(t) \exp(\beta_0^\top Z_1'(t))|}{Y_1(t) \exp(\beta_0^\top Z_1(t)) + \sum_{j=2}^n Y_j(t) \exp(\beta_0^\top Z_j(t))} \leq \frac{C_Z^2 \exp(2C_Z\|\beta_0\|_2)}{\sum_{j=1}^n Y_j(t)}
\end{align*}
where the first inequality is by \eqref{eq:deg2_bound}.

We therefore have that 
\begin{align*}
    (III) \leq \frac{1}{n} \sum_{i=1}^n \int_0^1 \left\{a_1(t) + a_2(t)\right\} \dint N_i(t) \leq  \frac{2 \exp(2C_Z\|\beta_0\|_2)C_Z^2(1 + \log (n))}{n}.
\end{align*}

\bigskip
\noindent \textbf{Term (IV):} we can write the integrand of $(IV)$ as 
    \begin{align*}
   &\frac{ \left|\mathrm{tr} \left\{S^{(1)}(D; t, \beta_0)^{\otimes 2} S^{(0)}(D'; t, \beta_0)^2 - S^{(1)}(D'; t, \beta_0)^{\otimes 2} S^{(0)}(D; t, \beta_0)^2 \right\} \right|}{S^{(0)}(D';t, \beta_0)^2 S^{(0)}(D; t, \beta_0)^2} \\
   &\leq \frac{ \left|\mathrm{tr} \left\{S^{(1)}(D; t, \beta_0)^{\otimes 2}  - S^{(1)}(D'; t, \beta_0)^{\otimes 2} \right\} \right|}{S^{(0)}(D;t, \beta_0)^2} + \frac{ \left|\mathrm{tr} \left\{S^{(1)}(D'; t, \beta_0)^{\otimes 2} \left(S^{(0)}(D;t, \beta_0)^2 - S^{(0)}(D'; t, \beta_0)^2\right) \right\} \right|}{S^{(0)}(D'; t, \beta_0)^2S^{(0)}(D;t, \beta_0)^2} \\
   &= b_1(t) + b_2(t).
\end{align*}

We first note that 
\begin{align*}
    S^{(1)}(D; t, \beta_0)^{\otimes 2} &= \frac{1}{n^2} \Bigg( Y_1(t) \exp(2\beta_0^\top Z_1(t)) Z_1(t)^{\otimes2} +  U_1(t)^{\otimes 2}  \\
    &\hspace{2cm} + Y_1(t) \exp(\beta_0^\top Z_1(t)) Z_1(t) U_1(t)^\top + U_1(t) Y_1(t) \exp(\beta_0^\top Z_1(t)) Z_1(t)^\top  \Bigg),
\end{align*}
where $U_1(t) = \sum_{i=2}^n Y_i(t) \exp(\beta_0^\top Z_i(t)) Z_i(t)$. We can therefore decompose the bound for $d_1(t)$ into bounding 
\begin{align*}
   b_1(t) &\leq \frac{\mathrm{tr}\left\{Y_1(t) \exp(2\beta_0^\top Z_1(t)) Z_1(t)^{\otimes 2} - Y_1'(t) \exp(2\beta_0^\top Z_1'(t)) Z_1'(t)^{\otimes 2} \right\}}{n^2 S^{(0)}(D; t, \beta_0)^2} + \frac{2 U_1(t)^\top V_1(t) }{n^2 S^{(0)}(D; t, \beta_0)^2} \\
   &= b_{11}(t) + b_{12}(t), 
\end{align*}
in which we introduce the notation 
\begin{equation*}
    V_1(t) = Y_1(t) \exp(\beta_0^\top Z_1(t)) Z_1(t) S^{(0)}(D'; t, \beta_0)^2 - Y_1'(t) \exp(\beta_0^\top Z_1'(t)) Z_1'(t) S^{(0)}(D; t, \beta_0)^2.
\end{equation*}

By similar reasoning leading to \eqref{eq:sens_log_example}, we may bound 
\begin{align*}
    \frac{1}{n} \sum_{i=1}^n \int_0^1 b_{11}(t) \dint N_i(t) \leq \frac{\exp(4C_Z\|\beta_0\|_2)C_Z^2}{n} \sum_{i=1}^n \frac{1}{i^2} \leq \frac{2\exp(4C_Z\|\beta_0\|_2)C_Z^2 }{n}. 
\end{align*}

We may bound $b_{12}(t)$ by 
\begin{align*}
    b_{12}(t) &\leq \frac{2 C_Z \|V_1(t)\|_2}{nS^{(0)}(D; t, \beta_0)} \\
    &= \frac{2C_Z}{n}  \left\| \frac{ Y_1(t) \exp(\beta_0^\top Z_1(t)) Z_1(t) S^{(0)}(D'; t, \beta_0)^2 - Y_1'(t) \exp(\beta_0^\top Z_1'(t)) Z_1'(t) S^{(0)}(D; t, \beta_0)^2}{S^{(0)}(D; t, \beta_0)}\right\|_2 \\
    &\leq \frac{2C_Z}{n} \left\|\frac{Y_1(t) \exp(\beta_0^\top Z_1(t)) Z_1(t) \{S^{(0)}(D'; t, \beta_0)^2-S^{(0)}(D; t, \beta_0)^2\}}{S^{(0)}(D; t, \beta_0)} \right\|_2 \\
    &\quad + \frac{2 C_Z}{n} \left\| \frac{S^{(0)}(D; t, \beta_0)^2 \left\{ Y_1'(t) \exp(\beta_0^\top Z_1'(t)) Z_1'(t) - Y_1(t) \exp(\beta_0^\top Z_1(t)) Z_1(t) \right\}}{S^{(0)}(D; t, \beta_0)}\right\|_2\\
    &\leq  \frac{2C_Z}{n}  \Bigg\| \left(2+\frac{Y_1'(t) \exp(\beta_0^\top Z_1'(t)) - Y_1(t) \exp(\beta_0^\top Z_1(t))}{\sum_{j=1}^n Y_j(t) \exp(\beta_0^\top Z_j(t))} \right) \\
    &\hspace{5cm} \times Y_1(t) \exp(\beta_0^\top Z_1(t)) Z_1(t) \left\{S^{(0)}(D'; t, \beta_0) -S^{(0)}(D; t, \beta_0) \right\} \Bigg\|_2 \\
    &\quad +  \frac{4 C_Z^2 \exp(C_Z\|\beta_0\|_2)}{n} S^{(0)}(t, \beta_0) \\
    &\leq \frac{2C_Z^2 \exp(C_Z \|\beta_0\|_2)}{n} \left( 2+ \frac{\exp(C_Z \|\beta_0\|_2)}{n S^{(0)}(D; t, \beta_0)}\right) \frac{\exp(C_Z \|\beta_0\|_2)}{n} + \frac{4 C_Z^2 \exp(C_Z\|\beta_0\|_2)}{n} S^{(0)}(t, \beta_0) , 
\end{align*}
where the first inequality is due to \eqref{eq:deg1_bound}. We can therefore bound
\begin{align*}
    \frac{1}{n}\sum_{i=1}^n \int_0^1 b_{12}(t) \dint N_i(t) &\leq \frac{4C_Z^2 \exp(2C_Z\|\beta_0\|_2)}{n^2} + \frac{2C_Z^2 \exp(4C_Z\|\beta_0\|_2 )(1+\log(n))}{n^3} \\
    &\quad +\frac{2 C_Z^2 \exp(2C_Z\|\beta_0\|_2)(n+1)}{n^2}.  
\end{align*}

We can bound $b_2(t)$ similarly:
\begin{align*}
    b_2(t) &\leq \mathrm{tr}\left\{ \frac{S^{(1)}(D'; t, \beta_0)^{\otimes 2}}{S^{(0)}(D;t, \beta_0)^2}\right\}  \frac{S^{(0)}(D; t, \beta_0)+S^{(0)}(D'; t, \beta_0)}{S^{(0)}(D; t, \beta_0)} \left|\frac{S^{(0)}(D; t, \beta_0) - S^{(0)}(D'; t, \beta_0) } {S^{(0)}(D; t, \beta_0)} \right|\\
    &\leq C_Z^2 \left(2 + \frac{ \exp(C_Z\|\beta_0\|_2)}{n} \right) \frac{\exp(C_Z\|\beta_0\|_2)}{nS^{(0)}(D; t, \beta_0)};
\end{align*}
the second inequality is due to \eqref{eq:deg1_bound} and \eqref{eq:sens_log_example}. Integrating against the jump process leads to
\begin{align*}
    \frac{1}{n} \sum_{i=2}^n \int_0^1 b_2(t) \dint N_i(t) \leq \frac{2C_Z^2 \exp(2C_Z \|\beta_0\|_2)(1+ \log(n))}{n} + \frac{C_Z^2 \exp(3C_Z\|\beta_0\|_2)(1+\log(n))}{n^2}. 
\end{align*}

We thus have 
\begin{align*}
    (IV) &\leq \frac{1}{n} \sum_{i=1}^n \int_0^1 \left\{b_{11}(t) + b_{12}(t) + b_2(t) \right\} \dint N_i(t) \\
&\leq\frac{C_Z^2}{n}
\Bigg\{\exp\!\left(2C_Z\|\beta_0\|_2\right)(4+2\log n)
+2\exp\!\left(4C_Z\|\beta_0\|_2\right)\\
&\hspace{1.5cm}+\frac{\exp\!\left(3C_Z\|\beta_0\|_2\right)(1+\log n)
+6\exp\!\left(2C_Z\|\beta_0\|_2\right)}{n}
+\frac{2\exp\!\left(4C_Z\|\beta_0\|_2\right)(1+\log n)}{n^2}
\Bigg\}.
\end{align*}

Combining the bounds on terms $(I), (II), (III)$, and $(IV)$ concludes the proof. 
\end{proof}

\subsection{Proof of \texorpdfstring{\Cref{lemma:trace_plugin}}{}}
\label[appendix]{app:trace_plugin}
    \begin{proof} The privacy guarantee for $T(D_1; \beta_0)$ follows from the bound on the sensitivity of $\mathrm{tr}\{H(D; \beta_0)\}$ in \Cref{lemma:trace_sensitivity}, \Cref{lemma:laplace_mechanism}, and post-processing properties of differential privacy \citep[e.g.~Proposition 2.1 in][]{dwork2014algorithmic}. The overall privacy guarantee with respect to $D=D_1 \cup D_2$ then follows from the parallel composition \citep[e.g.~Theorem 2 in][]{Smith2022ParallelComposition}. 

     For the statement on error guarantees with the plug-in estimator, we adapt the proof of \Cref{lemma:scoretest_theoretical} to use $T(D_1; \beta_0) = \mathrm{tr}\{G(\beta^*)\} + \xi$ in place of $\mathrm{tr}\{G(\beta^*)\}$ in the threshold for rejection \eqref{eq:beta_composite_threshold}. Here, $\xi$ is the error term in estimating the trace, and we note that $\mathrm{tr}\{G(\beta^*)\} + \xi \geq 0$ due to the truncating the estimator to be non-negative. 

The effect on the type-I error control of using $T(D; \beta_0)$ would appear as a change in \eqref{eq:scoretest_type1} in the proof of \Cref{lemma:scoretest_theoretical}.  Define the event
\begin{equation*}
    \mathcal{E} = \left\{|\xi| < \frac{\bar{c}_1 \log(dn)^2}{\sqrt{n}} + \frac{K(n, \beta_0)\log(1/\alpha)}{\epsilon} \right\} \subset \left\{|\xi| < \frac{c_1^2}{2d} \right\}.
\end{equation*}
Under the null hypothesis, by \Cref{lemma:trace_estimator}, we have that $\mathbb{P}(\mathcal{E}^c) \leq \alpha + \bar{c}_2/n$ for some absolute constant $\bar{c}_2$. Here, $c_1$ is the absolute constant from the proof of \Cref{lemma:scoretest_theoretical}; the set inclusion is by appropriate choice of absolute constants in the conditions of \Cref{lemma:scoretest_theoretical}.  

Replacing the theoretical $\mathrm{tr}\{G(\beta^*)\}$ in the threshold with with $T(D; \beta_0)$, conditional on $\mathcal{E}$,  \eqref{eq:scoretest_type1} in the proof of \Cref{lemma:scoretest_theoretical} becomes 
\begin{align*}    &\mathbb{P}\left(\|A_n\|_2 > \sqrt{\mathrm{tr}\{G(\beta^*)\} + \xi} + \frac{c_1}{2\sqrt{d}}\right) \\
 &\quad\leq \mathbb{P}\left( \left|\|A_n\|^2_2 - \mathbb{E}[\|A_n\|^2_2] \right| \geq  \sqrt{\frac{c_1 (\mathrm{tr}\{G(\beta^*)\} + \xi)}{d}} + \frac{c_1^2}{d} + \xi\right) \\
     &\quad\leq \frac{4\mathrm{tr}\{G(\beta^*)^2\} + c/n}{(\sqrt{2c_1 \{\mathrm{tr}\{G(\beta^*)\} + c_1^2/d\}/d} + c_1^2/(2d) )^2} \lesssim \frac{1}{c_1} 
\end{align*}
for large enough $n$ and $d$. 

For the type-II error control, the main change lies in 
\eqref{eq:beta_comp_shifted}. Defining $v= -\sqrt{n}G(\beta^*)(\beta_0 - \beta^*)$, $r=\|\beta_0-\beta^*\|_2$, and $q= c_5/\sqrt{d} + \kappa/(\sqrt{n} \epsilon)+ c_3\left(r^2 \sqrt{n} /d + r \right)$ as before, we consider instead 
\begin{align}
    &\mathbb{P}\left(\|A_n+v\|^2_2 \leq \left(\sqrt{\mathrm{tr}\{G(\beta^*)\} + \xi} +q \right)^2 \right) \leq \mathbb{P}\left(\|A_n\|^2_2 \leq \mathrm{tr}\{G(\beta^*)\} - \frac{c_5}{\sqrt{d}} \right) \notag\\
    &\hspace{4cm}+ \mathbb{P} \left( 2A_n^\top v \leq \left(\sqrt{\mathrm{tr}\{G(\beta^*)\} + \xi} + q \right)^2 - \|v\|_2^2 -  \mathrm{tr}\{G(\beta^*)\} + \frac{c_5}{\sqrt{d}} \right).
    \label{eq:beta_comp_shifted_plugin}
\end{align}
The first term of \eqref{eq:beta_comp_shifted_plugin} can be dealt with in the same way as in the proof of \Cref{lemma:scoretest_theoretical}. 

By \Cref{lemma:trace_sensitivity}, we have for some absolute constants $\bar{c}_1, \bar{c}_2, \bar{c}_1'>0$ that the event
\begin{equation*}
    \mathcal{E} = \left\{ |\xi| \leq \bar{c}_1 \left \{  \frac{\log(dn)^2}{\sqrt{n}} + r \left( \frac{d\log(nd)^2}{\sqrt{n}} + 1 \right) \right\} + \frac{K(n, \beta_0)\log(1/\alpha)}{\epsilon} \right\} \subset \left\{ |\xi| \leq \frac{c_1}{d} + \bar{c}_1' r\right\}
\end{equation*}
satisfies $\mathbb{P}(\mathcal{E}^c) \leq \alpha + \bar{c}_2/n$; the set inclusion is under the conditions in \Cref{lemma:scoretest_theoretical} with appropriate choice of absolute constants. Under $\mathcal{E}$ and recalling the notation $\kappa$ defined in \eqref{eq:kappa_notation}, we have that 
\begin{align*}
\xi + \sqrt{\mathrm{tr}\{G(\beta^*)\} + \xi} \left(\frac{1}{\sqrt{d}} +\frac{\kappa}{\sqrt{n} \epsilon} +  \frac{r^2 \sqrt{n}}{d} + r \right) +  \frac{1}{\sqrt{d}} +  \frac{\kappa^2}{n \epsilon^2} + \left(\frac{r^2\sqrt{n}}{d} + r \right)^2 \lesssim\frac{nr^2}{d^2},
\end{align*}
provided that $\epsilon \lesssim 1$.
The rest of the proof can proceed as in \Cref{lemma:scoretest_theoretical}. 
\end{proof}

\begin{lemma}
Let $T(D; \beta_0)$ be as defined in \eqref{eq:trace_estimator_def}. There exists absolute constants $c_1, c_2>0$ such that for any $\beta_0, \beta^* \in \mathbb{B}_{C_\beta}(0) \subset \mathbb{R}^d$ and $\alpha>0$, we have 
    \begin{align*}
        \mathbb{P}\Bigg( \left|T(D; \beta_0) - \mathrm{tr}\{G(\beta^*)\} \right| > &\frac{K(n, \beta_0)\log(1/\alpha)}{\epsilon} \\
        & +c_1 \left \{  \frac{\log(dn)^2}{\sqrt{n}} + \|\beta_0-\beta^*\|_2 \left( \frac{d\log(dn)^2}{\sqrt{n}} + 1 \right) \right\} \Bigg) \leq \alpha + \frac{c_2}{n},   
        \end{align*}
        where $K(n, \beta_0)$ is defined in \Cref{lemma:trace_sensitivity}. 
        \label[lemma]{lemma:trace_estimator}
\end{lemma}
\begin{proof}
We can show the probability bound via a union bound argument: 
\begin{align}
     &\mathbb{P}\Bigg( \left|T(D; \beta_0) - \mathrm{tr}\{G(\beta^*)\} \right| > c_1 \left \{  \frac{\log(dn)^2}{\sqrt{n}} + \|\beta_0-\beta^*\|_2 \left( \frac{d\log(dn)^2}{\sqrt{n}} + 1 \right) \right\} + \frac{K(n, \beta_0)\log(1/\alpha)}{\epsilon}  \Bigg) \notag \\
     &\quad \leq  \mathbb{P} \left(\left|\mathrm{tr}\left\{\ddot{\ell}_n(\beta_0)/n\right\} - \mathrm{tr}\left\{\ddot{\ell}_n(\beta^*)/n\right\} \right| \geq c_1 \|\beta_0-\beta^*\|_2 \left( \frac{d\log(dn)^2}{\sqrt{n}} + 1 \right)  \right) \notag\\
     &\quad \quad + \mathbb{P}\left(\left|\mathrm{tr}\left\{\ddot{\ell}_n(\beta^*)/n\right\} + \mathrm{tr}\{G(\beta^*)\} \right| \geq c_1 \frac{\log(dn)^2}{\sqrt{n}}\right) + \mathbb{P} \left( |W'| \geq \log(1/\alpha) \right) \notag\\
     &\quad = (I) + (II) + (III).    \label{eq:trace_estimator_triangle}
\end{align}

To bound $(I)$, we observe that with probability at least $1 - c/n$, we have that
\begin{align*}
    \frac{1}{n}\left|\mathrm{tr} \left\{\ddot{\ell}_n(\beta^*) - \ddot{\ell}_n(\beta_0) \right\} \right| &\leq \frac{\max \left\{|\exp(2C_Z\|\beta_0 - \beta^*\|_2) - 1|, |\exp(-2 C_Z \|\beta_0 - \beta^*\|_2) - 1| \right\}}{n} \mathrm{tr}\{\ddot{\ell}_n(\beta^*)\}\\
    &\lesssim \frac{d \|\beta_0-\beta^*\|_2}{n} \left\{ \left\|\ddot{\ell}_n(\beta^*) + nG(\beta^*) \right\| + \frac{n}{d} \right\} \\
    &\lesssim \|\beta_0-\beta^*\|_2 \left\{ \frac{d\log(dn)^2}{\sqrt{n}} + 1 \right\}
\end{align*}
The high probability event is for obtaining the final inequality and follows from \Cref{lemma:hessian_convergence}. The first inequality is by Lemma 3.2 of \cite{huang2013oracle} and the second is by the triangle inequality and \Cref{assp:eigenvalue}.  

We can bound the second term in \eqref{eq:trace_estimator_triangle} by $(II) \lesssim \frac{1}{n}$; the proof of this is deferred to \Cref{lemma:trace_convergence}. Finally, we have that $(III) \leq \alpha$ by the tail bound for a standard Laplace random variable. 
\end{proof}

\section{Technical details from \texorpdfstring{\Cref{sec-4}}{}}
\label[appendix]{app:sec-4}
\Cref{app:DP-NA_alg} contains the details of the cumulative hazard estimator used in \Cref{prop:cumulative_upper}. The proof of the upper and lower bounds can be found in Appendices \ref{app:cumulative_upper} and \ref{app:cumulative_lower} respectively. 

\subsection{Cumulative hazard estimator}
The private cumulative hazard estimator used in \Cref{prop:cumulative_upper} can be found in \Cref{alg:DP-NA}.
\label[appendix]{app:DP-NA_alg}
\begin{algorithm}[h!]
\caption{Differentially private Nelson--Aalen estimator}
\label{alg:DP-NA}
\textbf{Input}: Dataset $\{(T_{i}, \Delta_{ i})\}_{i = 1}^{n}$, privacy parameters $\epsilon, \delta >0$.

\begin{algorithmic}[1]
\State Set 
\begin{equation*}
    \hat{p} = \frac{1}{\lfloor 0.05n \rfloor} \sum_{i=1}^{\lfloor 0.05n \rfloor} Y_i(1) + \frac{\sqrt{2 \log(1.25/\delta)}}{n \epsilon}W, \quad W \sim \mathcal{N}(0, 1)
\end{equation*}
\State Set 
\begin{equation*}
    n' = n-\lfloor0.05n\rfloor, \quad c = 0.9\hat{p}, \quad\text{and }  h=\left \lfloor \frac{1}{2} \log_2 \left( \min \{n', (n')^2 \epsilon^2\}\right) \right \rfloor.
\end{equation*}
    \For{$m = 1, \ldots, 2^h$}
        \State Set 
            \begin{equation*}
                x_{h, m} = \int_{(m-1)/2^h}^{m/2^h} \sum_{i=\lfloor0.05n\rfloor+1}^{n} \frac{\dint  N_{i}(t)}{\max \left\{cn', \sum_{i=\lfloor0.05n\rfloor+1}^n Y_i(t) \right\}}
            \end{equation*}
        \For{$l = h-1, h-2, \ldots, 1$}
            \For{$m = 1, \ldots, 2^l$}
                \State Set $x_{l, m} = x_{ l+1, 2m-1} + x_{l+1, 2m}$
            \EndFor
        \EndFor
    \EndFor
    \For{$l = 1, \dots, h$}
        \For{$m = 1, \dots, 2^l$}
            \State Generate independently \begin{equation*}
                W_{l, m} \sim \mathcal{N}\left(0, \left[\frac{1}{c^4} + \frac{3}{c^2}\right]\frac{2 \log(1/\delta)/\epsilon + 1}{(n')^2 \epsilon/h} \right).
            \end{equation*}
            \State Set $x_{ l, m} = x_{ l, m} + W_{l, m}$.
        \EndFor
    \EndFor
\State For any $t \in [0, 1]$, let $(b_1, \dots b_h)$ be the binary representation of $\lfloor 2^h t \rfloor$ and set 
    \begin{equation*}
        \widehat{\Lambda}(t) =  \max \left\{0, \sum_{l=1}^h \mathbbm{1}\{b_l = 1\} x_{l, \, \sum_{k=1}^l 2^{l-k} b_k}\right\}.  
    \end{equation*} 
\State \textbf{Output}: $\widehat{\Lambda}(t)_{t \in [0, 1]}$.
\end{algorithmic}
\end{algorithm}

\subsection{Proof of \texorpdfstring{\Cref{prop:cumulative_upper}}{}}
\label[appendix]{app:cumulative_upper}
\begin{proof}
The privacy guarantees for $\widehat{\Lambda}_1$ and $\widehat{\Lambda}_2$ follow from Theorem 5 of \cite{FDPCox}, so we have that $\phi \in \Phi_{n_1, n_2, \epsilon_1, \epsilon_2, \delta_1, \delta_2}$.  

\bigskip
\noindent \textbf{Type-I error:} suppose that $\Lambda_1(t) = \Lambda_2(t)=\Lambda(t)$ for $t \in [0, 1]$, and let $r$ be as defined in \eqref{eq-two-sample-upper-rate}. By the triangle inequality, we have that 
\begin{equation*}
\left\{
\sup_{t \in [0,1]} |\widehat{\Lambda}_1(t) - \widehat{\Lambda}_2(t)| > \tau
\right\} \subseteq \left\{
\sup_{t \in [0,1]} |\widehat{\Lambda}_1(t) - \Lambda(t)| > \frac{\tau}{2} \right\}
\cup \left\{\sup_{t \in [0,1]} |\widehat{\Lambda}_2(t) - \Lambda(t)| > \frac{\tau}{2} \right\}.
\end{equation*}
By Markov's inequality, we have 
\begin{align}
\label{eq:hazard_Markov}
    \mathbb{P}\left( \sup_{t \in [0, 1]} |\widehat{\Lambda}_1(t) - \Lambda(t)|> \frac{\tau}{2}\right) &\leq \frac{2\mathbb{E} \left[\sup_{t \in [0, 1]} |\widehat{\Lambda}_1(t) - \Lambda(t)| \right]}{\tau} \notag\\ 
    &\leq  \frac{c'}{\tau} \left\{ \frac{1}{\sqrt{n_1}} + \frac{\log_2(\min\{n_1^{1/2}, n_1\epsilon_1\})^2 \log(1/\delta_1)}{n_1 \epsilon_1} \right\}
\end{align}
where the second inequality is from following the proof of Theorem 5 of \cite{FDPCox} for a single server setting. An analogous result holds for  $\widehat{\Lambda}_2$, so choosing appropriate absolute constants for defining $\tau$ in \eqref{eq:two_sample_threshold} leads to the type-I error bound. 

\bigskip
\noindent \textbf{Type-II error:} we now suppose that $\Lambda_1(t)$ and $\Lambda_2(t)$ satisfy $\sup_{t \in [0, 1]} |\Lambda_1(t) - \Lambda_2(t)| > r$. For any $t \in [0, 1]$ we have that
\begin{align*}
    |\widehat{\Lambda}_1(t) - \widehat{\Lambda}_2(t)| &\geq |\widehat{\Lambda}_1(t) - \Lambda_2(t)| - |\widehat{\Lambda}_2(t) - \Lambda_2(t)| \\
    &\geq |\Lambda_1(t) - \Lambda_2(t) | - | \widehat{\Lambda}_1(t) - \Lambda_1(t)|- |\widehat{\Lambda}_2(t) - \Lambda_2(t)|. 
\end{align*}
This pointwise inequality implies that 
\begin{equation*}
    \sup_{t \in [0, 1]} |\widehat{\Lambda}_1(t) - \widehat{\Lambda}_2(t)| \geq \sup_{t \in [0, 1]} |\Lambda_1(t) - \Lambda_2(t)| - \sup_{t \in [0, 1]} | \widehat{\Lambda}_1(t) - \Lambda_1(t)|- \sup_{t \in [0, 1]} |\widehat{\Lambda}_2(t) - \Lambda_2(t)|.
\end{equation*}
We therefore have that 
\begin{align*}
    \mathbb{P}_{\mathrm{H}_1} \left(\phi=0 \right) &= \mathbb{P} \left( \sup_{t \in [0, 1]} |\widehat{\Lambda}_1(t) - \widehat{\Lambda}_2(t)| \leq \tau\right) \\
    &\leq \mathbb{P} \left( \sup_{t \in [0, 1]} |\widehat{\Lambda}_1(t) - \Lambda_1(t)| + \sup_{t \in [0,1]}|\widehat{\Lambda}_2(t) - \Lambda_2(t)| \geq r - \tau \right)\\
    &\leq \mathbb{P} \left( \sup_{t \in [0, 1]} |\widehat{\Lambda}_1(t) - \Lambda_1(t)| \geq \frac{r - \tau}{2} \right) + \mathbb{P} \left( \sup_{t \in [0, 1]} |\widehat{\Lambda}_2(t) - \Lambda_2(t)| \geq \frac{r -\tau}{2} \right) \\
    &\leq \frac{2c'}{r-\tau} \Bigg\{ \frac{1}{\sqrt{n_1}} + \frac{\log_2(\min\{n_1^{1/2}, n_1\epsilon_1\})^2 \log(1/\delta_1)}{n_1 \epsilon_1}  \\
     &\hspace{3cm} +\frac{1}{\sqrt{n_2}}  + \frac{\log_2(\min\{n_2^{1/2}, n_2\epsilon_2\})^2 \log(1/\delta_2)}{n_2 \epsilon_2} \Bigg\}
\end{align*}
where the final line is derived in a similar way to \eqref{eq:hazard_Markov}. It follows that there exists some absolute constant $C$ such that for $r$ satisfying \eqref{eq-two-sample-upper-rate}, the type-II error can be controlled. 
\end{proof}

\subsection{Proof of \texorpdfstring{\Cref{prop:cumulative_lower}}{}}
\label[appendix]{app:cumulative_lower}
\begin{proof}
    Similar to \cite{arias2018remember}, we consider a reduction to the one-sample goodness-of-fit testing problem. For $\mathbb{P}_{\Lambda_1} \in \mathcal{C}$, we define the goodness-of-fit detection boundary as  
\begin{align*}
    \tilde{r}^*(\mathbb{P}_{\Lambda_1},  n, \epsilon, \delta) = \inf \bigg\{ r>0: \,
    &\exists\, (\epsilon, \delta)\text{-differentially private } \phi \text{ s.t. } \mathbb{P}_{\Lambda_1}^{\otimes n}(\phi=1) + \mathbb{P}_{\Lambda_2}^{\otimes n}(\phi=0) < \tfrac{1}{3} \\
    &\quad \forall \mathbb{P}_{\Lambda_2} \in \mathcal{C}: \sup_{t \in [0, 1]} |\Lambda_2(t) - \Lambda_1(t)| > r \bigg\};
\end{align*}
recalling that $\mathcal{C}$ is the set of distributions for the right-censored observations $(T, \Delta)$ on $\mathbb{R}_{\geq 0} \times \{0, 1\}$ that satisfy \Cref{assp:indep_atrisk} in a covariate-less setting.
We can bound the two-sample testing problem rate by that of the one-sample testing problem, the proof of which we defer to \Cref{subsec:onelessthantwo}.  
\begin{lemma}
For all $n_1, n_2 \in \mathbb{N}$ and $\epsilon_1, \epsilon_2 > 0$, $\delta_1, \delta_2 \geq 0$, it holds that 
\label[lemma]{lemma:onelessthantwo}
\begin{equation*}
    \sup_{\mathbb{P}_{\Lambda_1} \in \mathcal{C}} \max \left\{\tilde{r}^*(\mathbb{P}_{\Lambda_1}, n_1, \epsilon_1, \delta_1), \,\tilde{r}^*(\mathbb{P}_{\Lambda_1}, n_2, \epsilon_2, \delta_2) \right\} \leq r^*(n_1, n_2, \epsilon_1, \epsilon_2, \delta_1, \delta_2).
\end{equation*}
\end{lemma}
   By \Cref{lemma:onelessthantwo}, it suffices to show that 
    \begin{equation*}
        \tilde{r}^* \left(\mathbb{P}_{\Lambda_1}, n, \epsilon, \delta\right) \gtrsim \frac{1}{\sqrt{n}} + \frac{1}{n(\epsilon + \delta)}
    \end{equation*}
for some $\mathbb{P}_{\Lambda_1} \in \mathcal{C}$. Take $\Lambda_1(t) = t$ and $\Lambda_2(t) = \{1 + c(1/\sqrt{n} + 1/(n(\epsilon + \delta))\}t$ for some $c >0$ to be implicitly defined later. We will define $\mathbb{P}_{\Lambda_i}$ to be the distribution where failure times are distributed with cumulative hazard $\Lambda_i$ and censoring times as an independent $\mathrm{Exp}(1)$ random variable. 

Let $\phi$ be any $(\epsilon, \delta)$-differentially private mechanism. Suppose that $1/\sqrt{n} > 1/(n(\epsilon+\delta))$. It follows from a data-processing inequality and Pinsker's inequality that  
\begin{equation*}
    D_{\mathrm{TV}}\left(\phi\mathbb{P}_{\Lambda_1}^{\otimes n}, \phi\mathbb{P}_{\Lambda_2}^{\otimes n} \right) \leq  \sqrt{\frac{n}{2} D_{\mathrm{KL}}(\mathbb{P}_{\Lambda_1}, \mathbb{P}_{\Lambda_2})} \leq 2/3
\end{equation*}
for some small enough absolute constant $c>0$. The last bound is due to $D_{\mathrm{KL}}(\mathbb{P}_{\Lambda_1}, \mathbb{P}_{\Lambda_2}) \leq D_{\mathrm{KL}}(E_1, E_2)$ by a data-processing inequality and that
\begin{align*}
    D_{\mathrm{KL}}(E_1, E_2) &= \log\left( 1 + c \left(\frac{1}{\sqrt{n}} + \frac{1}{n (\epsilon+\delta)} \right) \right) + \frac{1}{1 + c(1/\sqrt{n} + 1/(n (\epsilon+\delta)))} - 1 \\
    &\leq \frac{4c^2}{ \min\{n, n^2 (\epsilon + \delta)^2\}} \lesssim \frac{1}{n}.
\end{align*}
Here, we write $E_i$ for the exponential distribution with cumulative hazard $\Lambda_i$ and the inequality is from 
\begin{equation*}
    \log(1+x) + \frac{1}{1+x} - 1 \leq \frac{x^2 + x + 1 - x - 1}{1+x} \leq x^2, \quad \forall x \geq 0.
\end{equation*}

For the other case where we have $1/(n(\epsilon+\delta)) > 1/\sqrt{n}$, by similar arguments to the proof of \Cref{prop:beta_binary_lower} and the bound on $D_{\mathrm{KL}}(\mathbb{P}_{\Lambda_1}, \mathbb{P}_{\Lambda_2})$ above, we have that 
\begin{align*}
    D_{\mathrm{TV}}(\phi \mathbb{P}_{\Lambda_1}^{\otimes n}, \phi\mathbb{P}_{\Lambda_2}^{\otimes n}) &\leq 2\exp(6n \epsilon D_{\mathrm{TV}}(\mathbb{P}_{\Lambda_1}, \mathbb{P}_{\Lambda_2})) n\delta D_{\mathrm{TV}}(\mathbb{P}_{\Lambda_1}, \mathbb{P}_{\Lambda_2}) \\
    &\quad+  \sqrt{3n \epsilon D_{\mathrm{TV}}(\mathbb{P}_{\Lambda_1}, \mathbb{P}_{\Lambda_2})\{\exp( 6n \epsilon D_{\mathrm{TV}}(\mathbb{P}_{\Lambda_1}, \mathbb{P}_{\Lambda_2})) - 1\}} \\
    &\leq 4 c \exp(12 c) + 12c,
\end{align*}
so we may bound $D_{\mathrm{TV}}(\phi \mathbb{P}_{\Lambda_1}^{\otimes n}, \phi \mathbb{P}_{\Lambda_2}^{\otimes n}) \leq 2/3$ by taking a sufficiently small absolute constant $c$. This concludes the proof since the sum of type-I and -II errors for any test $\phi$ is equal to $1- D_{\mathrm{TV}}(\phi \mathbb{P}_{\Lambda_1}^{\otimes n}, \phi \mathbb{P}_{\Lambda_2}^{\otimes n})$.
\end{proof}

\subsection{Proof of \texorpdfstring{\Cref{lemma:onelessthantwo}}{}}
\label[appendix]{subsec:onelessthantwo}
\begin{proof}
    Fix any $n_1, n_2 \in \mathbb{N}$ and $\epsilon_1, \epsilon_2 > 0$, $\delta_1, \delta_2 \geq 0$. Suppose that there exists $\mathbb{P}_{\Lambda_1} \in \mathcal{C}$ such that  
    \begin{equation*}
        \tilde{r}^*(\mathbb{P}_{\Lambda_1}, n_1, \epsilon_1, \delta_1) > r^*(n_1, n_2, \epsilon_1, \epsilon_2, \delta_1, \delta_2).
    \end{equation*}
    This implies that there exists some cumulative hazard $\Lambda_2$ satisfying $\sup_{t \in [0, 1]}|\Lambda_1 - \Lambda_2| \in (r^*, \tilde{r}^*)$ such that for all $(\epsilon_1, \delta_1)$-differentially private mechanisms $M'$, 
        \begin{equation*}
        D_{\mathrm{TV}} \left(M'(\mathbb{P}_{\Lambda_1}^{\otimes n_1}), M' (\mathbb{P}_{\Lambda_2}^{\otimes n_1}) \right) < 2/3.
    \end{equation*}
    However, since $\sup_{t \in [0, 1]}|\Lambda_1(t) - \Lambda_2(t)| > r^*$, there exists $(\epsilon_1, \delta_1)$-differentially private $M_1$ and $(\epsilon_2, \delta_2)$-differentially private $M_2$ such that 
    \begin{align*}
        1/3 &\geq \mathbb{P}_{\mathrm{H}_0}(\phi=1) + \mathbb{P}_{\mathrm{H}_1}(\phi=0)\\
        &=1 - D_\mathrm{TV} \left( \phi \left\{ M_1(\mathbb{P}_{\Lambda_1}^{\otimes n_1}) \otimes M_2(\mathbb{P}_{\Lambda_1}^{\otimes n_2}) \right\}, \phi \left\{M_1(\mathbb{P}_{\Lambda_2}^{\otimes n_1}) \otimes M_2(\mathbb{P}_{\Lambda_2}^{\otimes n_2}) \right\}\right) \\
        &\geq  1- D_{\mathrm{TV}} \left( M_1(\mathbb{P}_{\Lambda_1}^{\otimes n_1}), M_1(\mathbb{P}_{\Lambda_2}^{\otimes n_1}) \right),
    \end{align*}
    leading to a contradiction. Similarly, we have $ \sup_{\mathbb{P}_{\Lambda_1} \in \mathcal{C}} \tilde{r}^*(\mathbb{P}_{\Lambda_1}, n_2, \epsilon_2, \delta_2) \leq r^*(n_1, n_2, \epsilon_1, \epsilon_2, \delta_1, \delta_2)$.
\end{proof}

\section{Auxiliary lemmas}
\label[appendix]{app:aux_lemmas}
\begin{lemma}[Laplace mechanism, e.g.~Theorem 3.6 in \citealp{dwork2014algorithmic}] \label[lemma]{lemma:laplace_mechanism}
    Let $\mathcal{D}$ be the collection of all datasets and let $f:\mathcal{D} \rightarrow \mathbb{R}^d$. Define the \emph{sensitivity} of $f$ to be $\mathrm{sens}(f) = \sup_{D\sim D'} \|f(D) - f(D')\|_1$, where the supremum is over all pairs $D, D' \in \mathcal{D}$ such that $D$ and $D'$ differ in one entry. Then the mechanism 
    \begin{equation*}
        M(D) = f(D) + \frac{\mathrm{sens}(f)}{\epsilon} W, \quad W_i \overset{\mathrm{i.i.d.}}{\sim} \mathrm{Laplace}(\mu=0, b=1), \, i \in \{1, \dots, d\}
    \end{equation*}
    satisfies $(\epsilon, 0)$-differential privacy. 
\end{lemma}

\begin{lemma}[Corollary of Lemma 16 in \citealp{FDPCox}] Let $\dot{\ell}_n(\beta^*)$ be the score function at the true parameter value. We have that 
\begin{equation*}
        \mathbb{P}\left( \|\dot{\ell}_n(\beta^*) / \sqrt{n}\|_2 > u\right) \lesssim \exp ( -cu^2)
    \end{equation*}
    for some absolute constant $c$. 
\label[lemma]{lemma:true_grad_convergence}
\end{lemma}

\begin{lemma}[Corollary of Lemma 18 in \citealp{FDPCox}] Let $\ddot{\ell}_n(\beta^*)$ be the Hessian of the log partial-likelihood evaluated at the true parameter value. There exists an absolute constant $c_1>0$ such that 
\begin{equation*}
    \mathbb{P} \left(\left\| \ddot{\ell}_n(\beta^*) + n G^*(\beta^*) \right\| \leq c_1 {\log(dn)^2} n^{1/2} \right) \lesssim  \frac{1}{n},
\end{equation*}
where $G(\beta^*)$ is defined in \eqref{eq-G-def}. 
\label[lemma]{lemma:hessian_convergence}
\end{lemma}

\begin{lemma}
\label[lemma]{lemma:trace_convergence}
For the Hessian of the log partial-likelihood evaluated at the true parameter value $\ddot{\ell}_n(\beta^*)$, it holds that 
\begin{equation*}
    \mathbb{P}\left( \left| \mathrm{tr} \left\{\frac{1}{n}\ddot{\ell}_n(\beta^*) + G(\beta^*) \right\}\right| > \frac{c \log(dn)^2}{n^{1/2}}\right) \lesssim \frac{1}{n},
\end{equation*}
where $c>0$ is an absolute constant that does not depend on $n$ or $d$.

\end{lemma}
\begin{proof}
This proof is essentially the proof of Lemma 18 in \cite{FDPCox} with some minor modifications. We start by introducing the notation 
\begin{equation*}
G_n(t, \beta) = \frac{1}{n}\sum_{i=1}^n Y_i(t) \exp(Z_i(t)^\top\beta)[Z_i(t) - \mu(t, \beta)]^{\otimes 2}, \quad \mu(t, \beta) = \frac{\mathbb{E}\{Z(t) Y(t) \exp(\beta^\top Z(t))\}}
{\mathbb{E}\{Y(t) \exp(\beta^\top Z(t))\}},
\end{equation*}
and 
\begin{equation*}
    V_n(t, \beta) = \frac{\sum_{i=1}^{n} Y_{i}(t) \exp\{\beta^{\top} Z_{i}(t)\} \left\{Z_{i}(t) - \bar{Z}(t, \beta) \right\}^{\otimes 2}}{\sum_{i=1}^n Y_i(t) \exp(\beta^\top Z_i(t))},
    \end{equation*}
so that $\mathbb{E} \int_0^1 G_n(t, \beta^*) \dint \Lambda_0(t) = G(\beta^*)$ and $\ddot{\ell}_n(\beta) = -\sum_{i=1}^n \int_0^1 V_n(t, \beta) \dint N_i(t)$.

By the triangle inequality, we have that 
\begin{align}
    &  \left|\mathrm{tr} \left(\frac{1}{n}\ddot{\ell}_n(\beta^*) + \mathbb{E} \int_0^1 G_n(t, \beta^*)\,  d\Lambda_0(t) \right) \right| \nonumber \\
    &\leq  \left| \mathrm{tr} \left( \frac{1}{n}\sum_{i=1}^n \int_0^{1} V_n(t, \beta^*) \dint N_i(s)- \int_0^{1} V_n(t, \beta^*) S^{(0)}(t, \beta^*)\, \dint \Lambda_0(t) \right)\right| \nonumber \\
    & \quad+ \left| \mathrm{tr} \left(\int_0^{1} V_n(t, \beta^*) S^{(0)}(t, \beta^*) \dint \Lambda_0(t) - \int_0^{1} G_n(t, \beta^*) \dint \Lambda_0(t)  \right)\right|\nonumber \\
    & \quad+ \left| \mathrm{tr} \left(\int_0^{1} G_n(t, \beta^*) \dint \Lambda_0(t) - \mathbb{E} \int_0^{1} G_n(t, \beta^*)\dint \Lambda_0(t)  \right)\right| \nonumber \\
    &= (I) + (II) + (III),
    \label{eq:trace_triangle}
\end{align}
where $S^{(0)}(t, \beta^*) = \frac{1}{n} \sum_{i=1}^n Y_i(t) \exp\{Z_i(t)^\top \beta^*\}$. We will give high probability upper bounds for the three terms in \eqref{eq:trace_triangle}.

\bigskip
\noindent \textbf{Term (I).}  Since $V_n(t, \beta^*)$ is the weighted average of some $v_i v_i^\top$, where each $\|v_i\|_2 \leq 2C_Z$ by \Cref{assp:covariates}, we have that for all times $t \in [0, 1]$, $\mathrm{tr}\{V_n(t, \beta^*)\}$ is a weighted average of $\mathrm{tr}\{v_i v_i^\top\} = \mathrm{tr}\{v_i^\top v_i\} \leq 4C_Z^2$. It follows that  
\begin{equation*}
    \left| \mathrm{tr} \left(\frac{1}{n}\int_0^{1}\sum_{i=1}^nV_n(t, \beta^*) \dint N_i(s)- \int_0^{1} V_n(t, \beta^*) S^{(0)}(t, \beta^*)\, \dint \Lambda_0(t) \right)\right| \leq 4C_Z^2  \sup_{t\in [0, 1]} |M(t)|
\end{equation*}
where $M(t) = \frac{1}{n}\sum_{i=1}^n N_i(t) - \int_0
^t S^{(0)}(s, \beta^*) \dint \Lambda_0(s)$ is a martingale. Since the jump times 

\begin{equation*}
    \tau_j = \inf \left\{t \in [0, 1], \,\sum_{i=1}^n N_i(t)=j \right\}, \quad j \in \{1, \dots, n\}
\end{equation*}
are stopping times, by noting that $\sup_{t \in [0, 1]} |M(t)| \leq \sup_{j} |M(\tau_j)| + 1/n$ and considering the discrete martingale $\{M(\tau_j)\}_{j=1}^n$, we obtain 
\begin{align}
    & \mathbb{P} \bigg\{\left| \mathrm{tr} \left(\frac{1}{n}\sum_{i=1}^n \int_0^{1} V_n(t, \beta^*) \dint N_i(s)- \int_0^{1} V_n(t, \beta^*) S^{(0)}(t, \beta^*)\, \dint \Lambda_0(t) \right)\right| > 4C_Z^2 \left(x+ \frac{2}{n} \right) \bigg\} \notag\\
    &\quad \leq 2 \exp \left(\frac{-cnx^2}{1+x} \right). \label{eq:hessian_1}
\end{align}

\noindent \textbf{Term (II).} We can write 
\begin{align*}
& \left| \mathrm{tr}\left( \int_0^{1} V_n(t, \beta^*)\, S^{(0)}(t, \beta^*) \dint\Lambda_0(t) - \int_0^{1} G_n(t, \beta^*) \dint\Lambda_0(t) \right)\right| \\
&\quad= \left| \mathrm{tr} \left( \int_0^{1} \{\bar{Z}(t, \beta^*) - \mu(t, \beta^*)\}^{\otimes 2} S^{(0)}(t, \beta^*)\, \dint \Lambda_0(t) \right) \right|.
\end{align*}
Let $R_i = \exp(-C_Z \|\beta^*\|_2) Y_i(1)$ for $i \in \{1, \dots, n\}$, and let 
\begin{equation*}
    \xi(t) = S^{(0)}(t, \beta^*) \left\{\bar{Z}(t, \beta^*) - \mu(t, \beta^*)\right\} = \frac{1}{n}\sum_{i=1}^n Y_i(t) \exp(Z_i(t)^\top\beta^*)\left\{Z_i(t) - \mu(t, \beta^*)\right\}.
\end{equation*}
Since $S^{(0)}(t, \beta^*) \geq \frac{1}{n}\sum_{i=1}^n R_i$ for all $t \in [0, 1]$, we have that 
\begin{equation*}
    \left|\mathrm{tr} \left(\int_0^{1} \{\bar{Z}(t, \beta^*) - \mu(t, \beta^*)\}^{\otimes 2} S^{(0)}(t, \beta^*) \dint \Lambda_0(t) \right) \right| \leq \frac{\left|\mathrm{tr} \left(\int_0^{1}\xi(t)^{\otimes 2} \dint\Lambda_0(t) \right) \right|}{\frac{1}{n}\sum_{i=1}^n R_i}.
\end{equation*}
We use a union bound argument to control the numerator and denominator separately. For the denominator, by Hoeffding's inequality, we have that 
\begin{equation}
\mathbb{P}\left(\frac{1}{n} \sum_{i=1}^n Y_i(1) < p_0/2 \right) \leq \exp \left( \frac{-np_0^2}{2}\right).
\label{eq:hessian_denominator}
\end{equation}
For the numerator, we can bound 
\begin{equation*}
    \left| \mathrm{tr} \left(\int_0^{1} \xi(t)^{\otimes 2} \dint \Lambda_0(t) \right) \right|  \leq \Lambda_0(1) \sup_{t \in [0, 1]} ||\xi(t)||^2_2.
\end{equation*}
The covering arguments from the proof of Lemma 18 of \cite{FDPCox} show that for large enough $n$, we have that 
\begin{equation}
    \mathbb{P}\left(\sup_{t \in [0, 1]} \|\xi(t)\|_2^2 > \frac{c \log(dn)^2}{n} \right) \lesssim \frac{1}{n}.
    \label{eq:sup_process}
\end{equation}

Therefore, by \eqref{eq:hessian_denominator}, and \eqref{eq:sup_process}, we have for large enough $n$ that 
\begin{equation*}
   \mathbb{P} \left(\left| \mathrm{tr} \left(\int_0^{1}  \left(V_n(t, \beta^*) S^{(0)}(t, \beta^*)  - G_n(t, \beta^*) \right) \dint\Lambda_0(t) \right)\right|  > \frac{c_1 \log(dn)^2}{n} \right) \lesssim \frac{1}{n}. 
   \label{eq:hessian_2}
\end{equation*}

\noindent \textbf{Term (III).} Term $(III)$ can be written as $\frac{1}{n}|\sum_{i=1}^n Q_i - \mathbb{E}[Q_i]|$, where
\begin{equation*}
Q_i = \mathrm{tr} \left\{ \int_0^1 Y_i(t) \exp(Z_i(t)^\top \beta^*) \{Z_i(t) - \mu(t, \beta^*)\}^{\otimes 2} \dint \Lambda_0(t)\right\},
\end{equation*}
are i.i.d.~and $Q_i \in [0, 4C_Z^2 \exp(C_Z\|\beta^*\|_2) \Lambda_0(1)]$. By Hoeffding's inequality, we therefore have that 
\begin{equation*}
    \mathbb{P} \left(\left| \mathrm{tr} \left(\int_0^{1} G_n(t, \beta^*) \dint \Lambda_0(t) - \mathbb{E} \int_0^{1} G_n(t, \beta^*)\dint \Lambda_0(t)  \right)\right| \geq \frac{c_1 \log(n)}{\sqrt{n}}\right) \lesssim \frac{1}{n}.
\end{equation*}
\end{proof}

\begin{lemma}
    Let $X_1, \dots, X_n \in \mathbb{R}^d$ be i.i.d.~random vectors with $\mathbb{E}[X_1]=0$, $\|X_1\|_2 \leq C$ a.s.~and $\mathrm{Cov}(X_1) = G$. Define $S_n = \sum_{i=1}^n X_i / \sqrt{n}$. We have for any $v \in \mathbb{R}^d$, that 
    \begin{equation*}
        \mathbb{E}[\|S_n + v\|^2_2] = \mathrm{tr}(G) + \|v\|^2_2
    \end{equation*}
    and
    \begin{equation*}
        \mathrm{Var}(\|S_n + v\|^2_2) \leq 4\mathrm{tr}(G^2) + 8v^\top G v + 4C^4/n.
    \end{equation*}
    \label[lemma]{lemma:vector_expvar}
\end{lemma}
\begin{proof}
We have that 
\begin{align*}
    \mathbb{E}[\|S_n  + v\|^2_2] &= \frac{1}{n} \mathbb{E} \left[ \sum_{j=1}^d \left(\sum_{i=1}^n X_{ij} \right)^2\right] + \frac{2}{\sqrt{n}} \sum_{i=1}^n \mathbb{E}[X_i^\top v] + \|v\|^2_2 \\
    &= \sum_{j=1}^d \mathbb{E}[X_{1j}^2] + \|v\|^2_2 =\mathrm{tr}(G) + \|v\|^2_2.
\end{align*}

We may bound the variance by 
    \begin{align*}
        \mathrm{Var}(\|S_n + v\|^2_2) = \mathrm{Var}(\|S_n\|^2_2 + 2 S_n^\top v) &\leq 2\mathrm{Var}(\|S_n\|^2_2 ) + 8 \mathrm{Var}(S_n^\top v) \\
        &= 2 \mathrm{Var}(\|S_n\|^2_2) + 8 v^\top G v.
    \end{align*}
    
    The $\mathrm{Var}(\|S_n\|_2^2)$ term can be bound by
    \begin{align*}
        \mathrm{Var}(\|S_n\|^2_2) &
        =\mathrm{Var}\left( \frac{1}{n} \sum_{i=1}^n \|X_i\|^2_2 + \frac{1}{n} \sum_{i \neq j} X_i^\top X_j \right) \\
        &\overset{(a)}{\leq} \frac{2C^4}{n} + 2 \mathrm{Var} \left(\frac{1}{n} \sum_{i\neq j} X_i^\top X_j \right) \\
        &\overset{(b)}{=} \frac{2C^4}{n} + \frac{2(n-1)}{n} \mathrm{Var}(X_1^\top X_2) \\
        &\leq \frac{2C^4}{n} + 2\mathbb{E} \left[ (X_1^\top X_2)^2 \right] = \frac{2C^4}{n} + 2\mathbb{E} \left[ \mathrm{tr} \left\{ X_1 X_1^\top X_2 X_2^\top \right\}  \right] \\
        &= \frac{2C^4}{n} + 2\mathrm{tr} \left\{G^2 \right\}
    \end{align*}
where $(a)$ is due to the Cauchy--Schwarz inequality and that $\|X_i\|_2 \leq C$ a.s.~and $(b)$ is due to standard results for U-statistics \citep[e.g.][Chapter 1.3]{lee1990ustatistics}. 
\end{proof}

\begin{lemma}
 \label[lemma]{lemma:lrt_taylor}
    Let $f(x, y) = e^{x+y}/(e^x + e^y - 1)$, then 
    \begin{equation*}
        f(x, y) = 1 + xy - (x^2y + xy^2)/2 + O(|x|^4 + |y|^4).
    \end{equation*}
\end{lemma}
for $x, y$ close enough to 0. 
\begin{proof}
We use the Taylor expansion of $\exp(x)$, around zero and up to degree three, to obtain 
\begin{equation*}
    \exp(x+y) = 1+x+y+\frac{x^2+2xy+y^2}{2}
+\frac{x^3+3x^2y+3xy^2+y^3}{6}
+O((x+y)^4)
\end{equation*}
and substituting the Taylor expansions for 
\begin{equation*}
    \exp(x)+\exp(y)-1 = 1 + x + y + \frac{x^2+y^2}{2} + \frac{x^3 + y^3}{6} + O(x^4 + y^4) 
\end{equation*}
into the approximation $(1+q)^{-1} = 1-q+q^2-q^3 + O(q^4)$, we have
\begin{equation*}
    \frac{1}{e^x+e^y- 1} = 1-(x+y)+\frac{x^2+y^2}{2}+2xy
-\frac{x^3+y^3}{6} -2(x^2y + xy^2)
+O(|x|^4 + |y|^4).
\end{equation*}
Taking product and retaining all terms with degree less than or equal to three leads to 
    \begin{align*}
        f(x, y) 
        &= 1+ x y - \frac{1}{2}(x^2y + y^2x) + O(|x|^4 + |y|^4).
    \end{align*}
\end{proof}
\end{document}